\documentclass[12pt,reqno]{amsart}
\pdfoutput=1
\usepackage[utf8]{inputenc}
\usepackage[T1]{fontenc}
\usepackage[cal=boondoxo, scr=boondox]{mathalfa}
\usepackage{bookmark}
\usepackage{amsaddr}
\usepackage{cleveref, microtype, mathtools, mleftright, graphicx, csquotes}
\usepackage[hmargin=2.5cm, vmargin=3.5cm, paper=a4paper]{geometry}
\usepackage{lmodern, mathabx}
\usepackage{shuffle}

\usepackage[external]{forest}

\forestset{
  decor/.style = {
    label/.expanded = {[inner sep = 0.1ex, font=\unexpanded{\tiny}]right:{$#1$}}
  },
  root/.style = {minimum size = 0.5ex},
  decorated/.style = {
    for tree = {
      circle, fill, inner sep = 0ex, minimum size = 0.5ex,
      grow' = north, l = 0, l sep = 0.8ex, s sep = 0.5em,
      fit = tight, parent anchor = center, child anchor = center,
      delay = {decor/.option = content, content =}
    }
  },
  default preamble = {decorated, root},
  begin draw/.code={\begin{tikzpicture}[baseline={([yshift=-0.5ex]current bounding box.center)}]},
}

\author[N. Tapia]{Nikolas Tapia}
\email{tapia@wias-berlin.de}
\address{WIAS-Institut Berlin, Mohrenstr. 39., 10117 Berlin, Germany\\Technische Universität Berlin, Str. des 17. Juni 136, 10587 Berlin, Germany.}
\urladdr{https://www.wias-berlin.de/people/tapia}
\author[L. Zambotti]{Lorenzo Zambotti}
\email{zambotti@lpsm.paris}
\address{Laboratoire de Probabilités, Statistique et Modélisation, 4 place Jussieu, 75005 Paris, France}
\urladdr{http://www.lpsm.paris/pageperso/zambotti}
\title{The geometry of the space of branched rough paths}

\numberwithin{equation}{section}

\def\lrn#1{\left\vvvert#1\right\vvvert}
\newcommand{\1}{\mathbf{1}}
\newcommand{\g}{\mathfrak{g}}
\newcommand{\A}{\mathfrak{A}}

\newcommand{\C}{\mathcal{C}}
\newcommand{\F}{\mathcal{F}}

\newcommand{\N}{\mathbb{N}}
\newcommand{\R}{\mathbb R}
\newcommand{\T}{\mathcal{T}}

\newcommand{\p}{\mathfrak{p}}

\newcommand{\BR}{\bf{BRP}}

\def\dtR<#1>{\Forest{[#1]}}
\def\dtI<#1#2>{\Forest{[#1[#2]]}}
\def\dtII<#1#2#3>{\Forest{[#1[#2[#3]]]}}
\def\dtV<#1#2#3>{\Forest{[#1[#2][#3]]}}

\DeclareMathOperator{\BCH}{BCH}

\DeclareMathOperator{\id}{id}

\newtheorem{thm}{Theorem}[section]
\newtheorem{prp}[thm]{Proposition}
\newtheorem{dfn}[thm]{Definition}
\newtheorem{lmm}[thm]{Lemma}
\newtheorem{crl}[thm]{Corollary}
\theoremstyle{remark}
\newtheorem{rmk}[thm]{Remark}
\newtheorem{xmp}[thm]{Example}

\newfont{\indic}{bbmss12}

\mleftright
\let\bb\mathbb
\let\f\mathfrak
\let\cal\mathcal
\def\HH{\mathcal{H}}
\begin{document}

\begin{abstract}
We construct an explicit transitive free action of a Banach space of H\"older functions on the space of branched rough paths, which yields in particular a
bijection between theses two spaces.
This endows the space of branched rough paths with the structure of a principal homogeneous space over a Banach space and allows to characterize its automorphisms.
The construction is based on the Baker-Campbell-Hausdorff formula, on a constructive version of the Lyons-Victoir extension theorem and on the Hairer-Kelly map, which allows to describe branched rough paths in terms of anisotropic geometric rough paths.


\end{abstract}

\maketitle

\setlength{\parskip}{1.5ex}
\section{Introduction}

The theory of Rough Paths has been introduced by Terry Lyons in the '90s with the aim of giving an alternative construction of stochastic integration and stochastic differential equations.
More recently, it has been expanded by Martin Hairer to cover stochastic partial differential equations, with the invention of regularity structures.

A rough path and a \textsl{model} of a regularity structure are mathematical objects which must satisfy some algebraic and analytical constraints.
For instance, a rough path can be described as a H\"older function defined on an interval and taking values in a non-linear finite-dimensional Lie group; models of regularity structures are a generalization of this idea.
A crucial ingredient of regularity structures is the \textsl{renormalisation} procedure: given a family of regularized models, which fail to converge in an appropriate topology as the regularization is removed, one wants to modify it in a such a way that the algebraic and analytical constraints are still satisfied and the modified version converges.
This procedure has been obtained in \cite{BHZ,Chandra2017} for a general class of models with a \textsl{stationary} character.

The same question about rough paths has been asked recently in \cite{BCF,Bruned2017,Bruned2018}, and indeed it could have been asked much earlier.
Maybe this has not happened because the motivation was less compelling; although one can construct examples of rough paths depending on a positive parameter which do not converge as the parameter tends to $0$, this phenomenon is the exception rather than the rule.
However the problem of characterizing the automorphisms of the space of rough paths is clearly of interest; one
example is the transformation from It\^o to Stratonovich integration, see e.g. \cite{BA,EFP2,EFP1}. However our aim is to put this particular example in a much larger context.

We recall that there are several possible notions of rough paths; in particular we have \textsl{geometric} RPs and \textsl{branched} RPs, two notions defined respectively by Terry Lyons \cite{Lyons1998} and Massimiliano Gubinelli \cite{Gubinelli2010}, see \Cref{sse:RPthm,sse:bRP} below.
These two notions are intimately related to each other, as shown by Hairer and Kelly \cite{Hairer2014}, see Section \ref{sse:bRP} below.
We note that regularity structures \cite{Hairer2014d} are a natural and far-reaching generalization of branched RPs.

In this paper we concentrate on the automorphisms of the space of branched RPs, see below for a discussion of the geometric case.
Let $\F$ be the collection of all non-planar rooted forests with nodes decorated by $\{1,\dotsc,d\}$, see \Cref{sse:bRP} below.
For instance the following forest
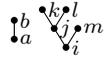
\[ \dtI<ab>\,\Forest{decorated[i[j[k][l]][m]]} \]
is an element of $\F$. We call $\T\subset\F$ the set of \textsl{rooted trees}, namely of \textsl{non-empty forests} with a single connected component. Grading elements $\tau\in\F$ by the number $|\tau|$ of their nodes we set
\[
\T_n\coloneq\{\tau\in\T: |\tau|\leq n\},
\qquad n\in\N.
\]
Let now $\HH$ be the linear span of $\F$.
It is possible to endow $\HH$ with a product and a coproduct $\Delta\colon\HH\to\HH\otimes\HH$ which make it a Hopf algebra, also known as the Butcher-Connes-Kreimer Hopf algebra, see \Cref{sec:brp} below.
We let $ G$ denote the set of all \emph{characters} over $\HH$, that is, elements of $G$ are functionals $X\in\HH^*$ that are also multiplicative in the sense that
\[\langle X,\tau\sigma\rangle=\langle X,\tau\rangle\langle X,\sigma\rangle \]
for all forests (and in particular trees) $\tau,\sigma\in\F$.
Furthermore, the set $G$ can be endowed with a product $\star$, dual to the coproduct, defined pointwise by $\langle X\star Y,\tau\rangle = \langle X\otimes Y,\Delta\tau\rangle.$
We work on the compact interval $[0,1]$ for simplicity, and all results can
be proved without difficulty on $[0,T]$ for any $T\geq 0$.

\begin{dfn}[Gubinelli \cite{Gubinelli2010}]\label{def:brp}
  Given $\gamma\in\,]0,1[$, a branched $\gamma$-rough path is a path $X\colon[0,1]^2\to G$
  which satisfies Chen's rule
  \[
  X_{su}\star X_{ut}=X_{st}, \qquad s,u,t\in[0,1],
  \]
  and the analytical condition
  \[
  |\langle X_{st},\tau\rangle|\lesssim|t-s|^{\gamma|\tau|}, \qquad \tau\in\F.
  \]
  Setting $x_t^i\coloneq\langle X_{0t},\dtR<i>\rangle$, $t\in[0,1]$, we say that $X$ is a branched $\gamma$-rough path over the path $x=(x^1,\dotsc,x^d)$. We denote by $\BR^\gamma$ the set of all branched $\gamma$-rough paths (for a fixed finite alphabet $\{1,\dotsc,d\}$).
\end{dfn}

By introducing the \emph{reduced coproduct} $\Delta'\colon\HH\to\HH\otimes\HH$
\[
\Delta'\tau\coloneq\Delta\tau-\tau\otimes\1-\1\otimes\tau,
\]
where $\1$ denotes the empty forest, Chen's rule can we rewritten as follows
\begin{equation}\label{redcop}
\delta \langle X,\tau\rangle_{sut} = \langle X_{su}\otimes X_{ut},\Delta'\tau\rangle,   \qquad s,u,t\in[0,1],
\end{equation}
where for $F\colon[0,1]^2\to\R$ we set $\delta F\colon[0,1]^3\to\R$,
\begin{equation}\label{eq:delta}
\delta F_{sut} \coloneq F_{st}-F_{su}-F_{ut},
\end{equation}
which is the second order finite increment considered by Gubinelli \cite{Gubinelli2003}. Note that
the right-hand side of \eqref{redcop} depends on the values of $X$ on trees with strictly fewer nodes
than $\tau$; if we can invert the operator $\delta$, then the right-hand side of \eqref{redcop}
determines the left-hand side. This is however not a trivial result. In fact, a simple (but crucial
for us) remark is the following: if $\gamma|\tau|\leq 1$, then for any $g^\tau\colon[0,1]\to\R$ such that $g^\tau\in C^{\gamma|\tau|}([0,1])$, the classical homogeneous Hölder space on $[0,1]$ with H\"older exponent $\gamma|\tau|$,
the function
\begin{equation}\label{non-un}
[0,1]^2\ni(s,t)\mapsto F_{st}\coloneq\langle X_{st},\tau\rangle+g^\tau_t-g^\tau_s
\end{equation}
also satisfies
\begin{equation}\label{non-un2}
\delta F_{sut}= \langle X_{su}\otimes X_{ut},\Delta'\tau\rangle, \qquad |F_{st}|\lesssim|t-s|^{\gamma|\tau|}, \qquad s,u,t\in[0,1].
\end{equation}
Inversely, if $F\colon[0,1]^2\to\R$ satisfies \eqref{non-un2}, then $F$ must satisfy \eqref{non-un} with
$g^\tau\in C^{\gamma|\tau|}([0,1])$.

If $\gamma|\tau|>1$, then Gubinelli's \textsl{Sewing Lemma} \cite{Gubinelli2003} yields that the function $(s,t)\mapsto\langle X_{st},\tau\rangle$ is uniquely determined by \eqref{non-un2} i.e. by the values of $X$ on trees with at most $|\tau|-1$ nodes, and therefore, applying a recursion, on trees with at most $N\coloneq\lfloor\gamma^{-1}\rfloor$ nodes.
More explicitly, the Sewing Lemma is an \textsl{existence and uniqueness result} for $[0,1]^2\ni(s,t)\mapsto \langle X_{st},\tau\rangle$ with $\gamma|\tau|>1$, once the right-hand side of \eqref{redcop} is known. However, for $\gamma|\tau|\leq 1$ we have
no uniqueness, as we have already seen, and existence is not trivial.

As we have seen in \eqref{non-un}, the value of $\langle X,\tau\rangle$ can be modified by adding
the increment of a function in $C^{\gamma|\tau|}([0,1])$, as long as $\gamma|\tau|\leq 1$. It seems
reasonable to think that it is therefore possible to construct an action on the set of branched $\gamma$-rough paths of the abelian group (under pointwise addition)
\[
  \cal C^\gamma\coloneq\{(g^\tau)_{\tau\in \T_N}: \, g^\tau_0=0, \, g^\tau\in C^{\gamma|\tau|}([0,1]), \, \forall\, \tau\in\T, |\tau|\leq N\},
\]
namely the set of all collections of functions $(g^\tau\in C^{\gamma|\tau|}([0,1]):\tau\in\T, |\tau|\leq N)$ indexed by rooted trees with fewer than $N\coloneq\lfloor\gamma^{-1}\rfloor$ nodes, such that $g^\tau_0=0$ and $g^\tau\in C^{\gamma|\tau|}([0,1])$. This is indeed the content of the following

\begin{thm}  \label{thm:intro.trans}
Let $\gamma\in\,]0,1[$ such that $\gamma^{-1}\not\in\N$.
  There is a transitive free action of $\C^\gamma$ on $\BR^\gamma$, namely a map $(g,X)\mapsto gX$ such that
  \begin{enumerate}
\item  for each $g,g'\in\cal C^\gamma$ and $X\in\BR^\gamma$ the identity $g'(gX) = (g+g')X$ holds.
\item if $(g^\tau)_{\tau\in \T_N}\in\cal C^\gamma$ is such that there exists a unique $\tau\in\T_N$
with $g^\tau\not\equiv 0$, then
\[
  \langle (gX)_{st},\tau\rangle=\langle X_{st},\tau\rangle+g^\tau_t-g^\tau_s
\]
and $\langle gX,\sigma\rangle=\langle X,\sigma\rangle$ for all $\sigma\in\T$ not containing $\tau$.
\item For every pair $X,X'\in\BR^\gamma$ there exists a \emph{unique} $g\in\C^\gamma$ such that $gX=X'$.
\end{enumerate}
\end{thm}
We say that a tree $\sigma\in\T$ {\it contains} a tree $\tau\in\T$ if there exists a subtree $\tau'$ of $\sigma$, not necessarily containing the root of $\sigma$, such that $\tau$ and $\tau'$ are
isomorphic as rooted trees, where the root of $\tau'$ is its node which is closest to the root of
$\sigma$.
Note that every $(g^\tau)_{\tau\in \T_N}\in\cal C^\gamma$ is the sum of finitely many elements of $\cal C^\gamma$ having satisfying the
property required in point (2) of \Cref{thm:intro.trans}.

If $\gamma>1/2$ then the result of \Cref{thm:intro.trans} is trivial. Indeed, in this case $N=1$, $\T_N=\{\dtR<i>: \, i=1,\ldots,d\}$, and
$\cal C^\gamma=\{g^{\dtR<i>}\in C^{\gamma}([0,1]): \, g^{\dtR<i>}_0=0, \, i=1,\ldots,d\}$. Then
the action is
\begin{equation}\label{eq:g1}
(g,X)\mapsto gX, \qquad \langle (gX)_{st},\dtR<i>\rangle \coloneq \langle X_{st},\dtR<i>\rangle
+ g^{\dtR<i>}_t-g^{\dtR<i>}_s,
\end{equation}
while the value of $\langle gX,\tau\rangle$ for $|\tau|\geq 2$ is uniquely determined by \eqref{redcop} via the Sewing Lemma. For example
\begin{equation}\label{eq:inte2}
\langle (gX)_{st},\dtI<ij>\rangle\coloneq\int_s^t (x^j_u-x^j_s+g^{\dtR<j>}_u-g^{\dtR<j>}_s) \, {\rm d}(x^i_u+g_u^{\dtR<i>}),
\end{equation}
where $x^i_u\coloneq\langle X_{0u},\dtR<i>\rangle$ and the integral is well-defined in the Young sense, see \cite[section 3]{Gubinelli2003}.

If $1/3<\gamma\leq 1/2$ then $N=2$ and $\T_2=\T_1\sqcup\{\dtI<ij>: \, i,j=1,\ldots,d\}$. Then
the action at level $|\tau|=1$ is still given by \eqref{eq:g1}, while at level $|\tau|=2$ we must have by \eqref{redcop}
\begin{equation}\label{eq:delta1} \delta \langle gX,\dtI<ij>\rangle_{sut} = \langle (gX)_{su}\otimes (gX)_{ut},\Delta'\tau\rangle = (x^j_u-x^j_s+g^{\dtR<j>}_u-g^{\dtR<j>}_s)(x^i_t-x^i_u+g^{\dtR<i>}_t-g^{\dtR<i>}_u).
\end{equation}
Although the right-hand side of \eqref{eq:delta1} is explicit and simple, in this case there is no
canonical choice for $\langle gX,\dtI<ij>\rangle$. An expression like \eqref{eq:inte2} is ill-defined in the Young sense, and the same is true if
we try the formulation
\begin{equation}\label{eq:inte}
\langle (gX)_{st},\dtI<ij>\rangle=\langle X_{st},\dtI<ij>\rangle+\int_s^t \left((x^j_u-x^j_s+g^{\dtR<j>}_u-g^{\dtR<j>}_s) \, {\rm d}g_u^{\dtR<i>}+ (g_u^{\dtR<j>}-g_s^{\dtR<j>}) \, {\rm d}x_u^i\right),
\end{equation}
which satisfies formally \eqref{eq:delta1}, but the Young integrals are ill defined since
$2\gamma\leq 1$. The construction of $\langle gX,\dtI<ij>\rangle$ is therefore not trivial in this case.

The same argument applies for any $\gamma\leq 1/2$ and any tree $\tau$ such that $2\leq |\tau|\leq N=\lfloor\gamma^{-1}\rfloor$, and the fact that the above Young integrals are not well defined shows why existence of the map $X\to gX$ is not trivial.

Since \Cref{thm:intro.trans} yields an action of $\C^\gamma$ on $\BR^\gamma$ which is \emph{regular}, i.e. free and transitive,
then $\BR^\gamma$ is a \emph{principal $\C^\gamma$-homogeneous space} or \emph{$\C^\gamma$-torsor}. In particular, $\BR^\gamma$ is a copy of $\C^\gamma$, but there is no canonical choice of an origin in $\BR^\gamma$.

Therefore, \Cref{thm:intro.trans} also yields the following
\begin{crl}\label{crl:param}
Given a branched $\gamma$-rough path $X$, the map $g\to gX$ yields a bijection between $\C^\gamma$ and the set of branched $\gamma$-rough paths.
\end{crl}
Therefore \Cref{crl:param} yields a complete parametrization of the space of branched rough paths. This result
is somewhat surprising, since rough paths form a non-linear space, in particular because of the Chen relation; however \Cref{crl:param} yields a natural bijection between the space of branched $\gamma$-rough paths and the
linear space $\C^\gamma$.

\Cref{crl:param} also gives a complete answer to the question of existence and characterization of branched $\gamma$-rough paths over a $\gamma$-H\"older path $x$. Unsurprisingly, for our construction we start from a result of T. Lyons and N. B. Victoir's \cite{Lyons2007} of 2007, which
was the first general theorem of existence of a geometric $\gamma$-rough path over a $\gamma$-H\"older path $x$, see our discussion of \Cref{thm:LV} below.

An important point to stress is that the action constructed in \Cref{thm:intro.trans} is neither
unique nor canonical. In the proof of \Cref{thm:charext} below, some parameters have to be fixed arbitrarily,
and the final outcome depends on them, see \Cref{rmk:change}. In this respect, the situation is similar to what happens in regularity structures with the reconstruction operator on spaces ${\mathcal D}^\gamma$ with a negative
exponent $\gamma<0$, see \cite[Theorem 3.10]{Hairer2014d}.

\subsection{Outline of our approach}
A key point in \Cref{thm:intro.trans} is the construction of branched $\gamma$-rough paths.
In the case of {\it geometric} rough paths, see Definition \ref{def:geo}, the \emph{signature} \cite{Chen1977, Lyons1998} of a smooth path $x\colon[0,1]\to\R^d$ yields a canonical construction. Other cases where geometric rough paths over non-smooth paths have been constructed are Brownian motion and fractional Brownian motion (see \cite{Coutin2002} for the case $H>\tfrac14$ and \cite{Nualart2011} for the general case) among others.
However, until T. Lyons and N. B. Victoir's paper \cite{Lyons2007} in 2007, this question remained largely open in the general case.
The precise result is as follows
\begin{thm}[Lyons--Victoir extension]
  If $p\in[1,\infty)\setminus\N$ and $\gamma\colon=1/p$, a $\gamma$-H\"older path $x\colon[0,1]\to\R^d$ can be lifted to a geometric $\gamma$-rough path. For any $p\geq 1$ and $\varepsilon\in\,]0,\gamma[$, a $\gamma$-H\"older path can be lifted to a geometric $(\gamma-\varepsilon)$-rough path.
  \label{thm:LV}
\end{thm}

Our first result is a version of this theorem which holds for rough paths in a more general algebraic context, see Theorem \ref{thm:charext} below. We use the Lyons-Victoir approach and an explicit form of the Baker--Campbell--Hausdorff formula by Reutenauer \cite{Reutenauer1986}, see formula \eqref{eq:BCHk} below. Whereas Lyons and Victoir used in one passage the axiom of choice, our method is completely constructive.

Using the same idea we extend this construction to the case where the collection $(x^1,\dotsc,x^d)$ is allowed to have different regularities in each component, which we call \emph{anisotropic (geometric) rough paths (aGRP)}, see Definition \ref{def:ageo}.
\begin{thm}
To each collection $(x^i)_{i=1,\dotsc,d}$, with $x^i\in C^{\gamma_i}([0,1])$, we can associate an anisotropic rough path $\bar X$ over $(x^i)_{i=1,\dotsc,d}$. For every collection $(g^i)_{i=1,\dotsc,d}$, with $g^i\in C^{\gamma_i}([0,1])$, denoting by
$g\bar X$ the anisotropic geometric rough path over $(x^i+g^i)_{i=1,\dotsc,d}$, we have
\[
g'(g\bar X)=(g+g')\bar X.
\]
  \label{thm:intro.aGRP}
\end{thm}
This kind of extension to rough paths has already been explored in the papers \cite{Boedihardjo2017, Gyurko2016} in the context of isomorphisms between geometric and branched rough paths.
It turns out that the additional property obtained by our method enables us to explicitly describe the propagation of suitable modifications from lower to higher degrees.

We then go on to describe the interpretation of the above results in the context of branched rough paths. The main tool is the Hairer--Kelly map \cite{Hairer2014}, that we introduce and describe
in \Cref{lmm:HK} and then use to encode branched rough paths via anisotropic geometric rough paths, along the same lines as in \cite[Theorem 4.3]{Boedihardjo2017}.
\begin{thm}\label{thm:1.6}
  Let $X$ be a branched $\gamma$-rough path. There exists an anisotropic geometric rough path $\bar{X}$ indexed by words on the alphabet $\T_N$, with exponents $(\gamma_\tau=\gamma|\tau|, \tau\in\T_N)$, and such that $\langle X,\tau\rangle = \langle\bar{X},\psi(\tau)\rangle$, where $\psi$ is the Hairer--Kelly map.
\end{thm}
The main difference of this result with \cite[Theorem 1.9]{Hairer2014} is that we obtain an anisotropic geometric rough path instead of a classical geometric rough path.
This means that we do not construct unneeded components, i.e. components with regularity larger than 1, and we also obtain the right Hölder estimates in terms of the size of the indexing tree.
This addresses two problems mentioned in Hairer and Kelly's work, namely Remarks 4.14 and 5.9 in \cite{Hairer2014}.

We then use \Cref{thm:intro.aGRP} and \Cref{thm:1.6} to construct our action on branched rough
paths. Given $(g,X)\in\cal C^\gamma\times\BR^\gamma$, we construct the anisotropic geometric rough
paths $\bar X$ and $g\bar X$ and then define the branched rough path $gX\in\BR^\gamma$ as $\langle gX,\tau\rangle = \langle g\bar{X},\psi(\tau)\rangle$, where $\psi$ is the Hairer--Kelly map.

Our approach also does not make use of Foissy-Chapoton's Hopf-algebra isomorphism \cite{Chapoton2010,Foissy2002} between the Butcher--Connes--Kreimer Hopf algebra and the shuffle algebra over a complicated set $I$ of trees as is done in \cite{Boedihardjo2017}.
This allows us to construct an action of a \textsl{larger} group on the set of branched rough paths;
indeed, using the above isomorphism one would obtain a transformation group parametrized by $(g^\tau)_{\tau\in I}$ where $I$ is the aforementioned set of trees of Foissy-Chapoton's results and $g^\tau\in C^{\gamma|\tau|}$;
on the other hand our approach yields a transformation group parametrized by $(g^\tau)_{\tau\in\T_N}$. With the smaller set $I\cap\T_N$, transitivity of the action $g\mapsto gX$ would be lost.

Finally we note that we use a special property of the Butcher-Connes-Kreimer Hopf algebra: the fact that it is freely
generated as an algebra by the set of trees, so defining characters over it is significantly easier than in the geometric case.
  To define an element $X\in G$ it suffices to give the values $\langle X,\tau\rangle$ for all trees $\tau\in\T$; by freeness there is a unique multiplicative extension to all of $\HH$.
  This is not at all the case for geometric rough paths: the shuffle algebra $T(A)$ over an alphabet $A$ is {not} free over the linear span of words so if one is willing to define a character $X$ over $T(A)$ there are additional algebraic constraints that the values of $X$ on words must satisfy.

\paragraph{\textbf{Outline.}} We start by reviewing all the theoretical concepts needed to make the exposition in this section formal.
In \Cref{sse:RPthm} we state and prove the main result of this chapter. We extend the notion of rough path and we give an explicit construction of such a generalized rough path above any given path $x\in C^\gamma$.
Next, in \Cref{sse:aRP} we extend this result to the class of anisotropic geometric rough paths.
Finally, in \Cref{sse:bRP} we connect our construction with M. Gubinelli's branched rough paths, and we extend M. Hairer and D. Kelly's work in \Cref{sss:RPHK}.
We also explore possible connections with renormalisation in \Cref{sss:mod} by studying how our construction behaves under modification of the underlying paths.
Then, we connect this approach with a recent work by Bruned, Chevyrev, Friz and Preiß \cite{Bruned2017} in \Cref{sss:BCFP}, who borrowed ideas from the theory of Regularity Structures \cite{BHZ,Hairer2014d} and proposed a renormalisation procedure for geometric and branched rough paths \cite{Bruned2017}
based on pre-Lie morphisms.

The main difference between our result and the BCFP procedure is that they consider translation only by time-independent factors, whereas --under reasonable hypotheses-- we are also able to handle general translations depending on the time parameter.
We also mention that some further algebraic aspects of renormalisation in rough paths have been recently developed in \cite{Bruned2018}.

\paragraph{\textbf{Acknowledgements.}} The authors thank Jean-David Jacques for pointing out a mistake in a previous version. 
N.T. acknowledges support by the CONICYT/Doc\-to\-ra\-do Na\-cio\-nal doctoral scholarship grant number 2013-21130733, \emph{Núcleo Milenio Modelos Estocásticos de Sistemas Complejos y Desordenados} and the Berlin Mathematical School MATH+ EF1-5 project ``On robustness of deep neural networks''. L.Z. gratefully acknowledges support by the project of the Agence Nationale de la Recherche ANR-15-CE40-0020-01 grant LSD.

\section{Preliminaries}\label{sec:2}
A Hopf algebra $\HH$ is a vector space endowed with an \emph{associative product} $m\colon\HH\otimes\HH\to\HH$:
\[
m(m\otimes\id)=m(\id\otimes m), 
\]
and a \emph{coassociative coproduct} $\Delta\colon\HH\to\HH\otimes\HH$:
\[
({\id}\otimes\Delta)\Delta=(\Delta\otimes{\id})\Delta,
\] 
satisfying moreover certain compatibility assumptions; $\HH$ is also supposed to have a \emph{unit} $\1\in\HH$, a \emph{counit} $\varepsilon\in\HH^*$ and an \emph{antipode} $S\colon\HH\to\HH$ such that
\[ m({\id}\otimes S)\Delta x=\varepsilon(x)\1 = m(S\otimes{\id})\Delta x \]
for all $x\in\HH$.
As usual we will use the more compact notation $m(x\otimes y)=xy$.
The reader is referred to the papers \cite{Cartier2007,Manchon2008} for further details.

\begin{dfn}
  We say that the Hopf algebra $\HH$ is \emph{graded} if it can be decomposed as a direct sum
\begin{equation}\label{eq:grade} \HH=\bigoplus_{n=0}^\infty\HH_{(n)} \end{equation}
with
\begin{equation}\label{eq:coprod}
m\colon\HH_{(n)}\otimes\HH_{(m)}\to\HH_{(n+m)}, \qquad \Delta\colon\HH_{(n)} \to \bigoplus_{p+q=n}\HH_{(p)}\otimes\HH_{(q)}.
\end{equation}
\end{dfn}
In a graded Hopf algebra, each element $x\in\HH$ can be decomposed as a sum
\begin{equation}\label{eq:hhx} x = \sum_{n=0}^\infty x_n, \qquad x_n\in\HH_{(n)}, \end{equation}
where only a finite number of the summands are non-zero.
We call each $x_n$ the \emph{homogeneous part of degree $n$} of $x$, and elements of $\HH_{(n)}$ are said to be homogeneous of degree $n$.
In this case we write $|x_n|=n$.

\begin{dfn}
  The graded Hopf algebra $\HH$ is \emph{connected} if the degree 0 part is one-dimensional. It is locally finite if $\dim\HH_{(n)}<\infty$ for all $n\ge 0$.
  \label{dfn:locfin}
\end{dfn}

From now on we consider a graded connected locally finite Hopf algebra $\HH$.
Then, for any homogeneous element $x\in\HH_{(n)}$ the coproduct can be written as
\[ \Delta x = x\otimes\1+\1\otimes x+\Delta'x, \qquad {\rm where} \quad
  \Delta'x \in \bigoplus_{\substack{p+q=n\\p,q\ge 1}}\HH_{(p)}\otimes\HH_{(q)}
  \]
and $\Delta'\colon\HH\to\HH\otimes\HH$ is known as the \emph{reduced coproduct}.
Furthermore, the coassociativity of $\Delta$ and of $\Delta'$, i.e. the identity $(\Delta'\otimes{\id})\Delta'=({\id}\otimes\Delta')\Delta'$, allows to unambiguously define their iterates $\Delta_n,\Delta_n'\colon\HH\to\HH^{\otimes (n+1)}$ by setting for $n\geq 2$
\[ \Delta_{n} = ({\id}\otimes\Delta_{n-1})\Delta, \qquad \Delta'_{n} = ({\id}\otimes\Delta'_{n-1})\Delta'. \]
Then we have, for a homogeneous element $x\in\HH_{(k)}$ of degree $k$,
\[ \Delta'_nx \in \smashoperator[r]{\bigoplus_{\substack{p_1+\dotsc+p_{n+1}=k\\p_j\ge 1}}}\HH_{(p_1)}\otimes\dotsm\otimes\HH_{(p_{n+1})}. \]
\begin{rmk}
  These properties of the iterated coproduct imply that the bialgebra $(\HH,\Delta)$ is \emph{conilpotent}, that is, for each homogeneous $x\in\HH_{(k)}$ there is an integer $n\leq k$ such that
  $\Delta_n'x=0$.
  We obtain also the inclusion
  \[ \Delta_{n}'\HH_{(n+1)}\subset\HH_{(1)}^{\otimes(n+1)}, \]
  that is, the $n$-fold reduced coproduct of a homogeneous element of degree $n+1$ is a sum of $(n+1)$-fold tensor products of homogeneous elements of degree 1.
  \label{rmk:nfoldcoprod}
\end{rmk}

We recall that in general the dual space $\HH^*$ carries an algebra structure given by the convolution product $\star$, dual to the coproduct $\Delta$, defined by
\[ \langle f\star g,x\rangle\coloneq\langle f\otimes g,\Delta x\rangle. \]
For a collection of maps $f_1,\dotsc,f_k\in\HH^*$ we have the formula
\begin{equation}
  f_1\star\dotsm\star f_k = 
  (f_1\otimes\dotsm\otimes f_k)\circ\Delta_{k-1}.
  \label{eq:kfoldconv}
\end{equation}
\begin{dfn}
  A \emph{character} on $\HH$ is a non-zero linear map $X\colon\HH\to\R$ 
\[ \langle X,xy\rangle = \langle X,x\rangle\langle X,y\rangle, \qquad \forall \, x,y\in\HH. \]
for all $x,y\in\HH$. We call $G$ the set of all characters on $\HH$.
An \emph{infinitesimal character} (or derivation) on $\HH$ is a linear map $\alpha\colon\HH\to\R$ such that
\[ \langle \alpha,xy\rangle = \langle \alpha,x\rangle\langle\varepsilon,y\rangle+\langle\varepsilon,x\rangle\langle \alpha,y\rangle, \qquad \forall \, x,y\in\HH. \]
We call $\g$ the set of all infinitesimal characters on $\HH$.
\end{dfn}
We observe that necessarily $\langle X,\1\rangle=1$ and $\langle\alpha,\1\rangle=0$ for all $X\in G$ and $\alpha\in\g$.
It is well known that the $(G,\star,\varepsilon)$ is a group with product $\star$, unit $\varepsilon$ and inverse $X^{-1}=X\circ S$ where $S$ is the antipode defined above.
Moreover $(\g,[\cdot,\cdot])$ is a Lie algebra with bracket $[\alpha,\beta]\coloneq\alpha\star \beta-\beta\star\alpha$.
See e.g. \cite{Manchon2008}.

\subsection{Nilpotent Lie algebras}
From \eqref{eq:coprod} we have
\begin{lmm}
  For any $N\in\N$ the subspace \[ \HH_N\coloneq \bigoplus_{k=0}^N\HH_{(k)}\subset \HH \] is a counital subcoalgebra of $(\HH, \Delta, \varepsilon)$.
  \label{lmm:subc}
\end{lmm}
By Lemma \ref{lmm:subc} we can consider the dual algebra $(\HH_N^*,\star,\varepsilon)$.
This algebra is also graded and connected, since we have the natural grading
\begin{equation}
  \HH_N^*=\bigoplus_{k=0}^N\HH_{(k)}^*, \qquad \HH_N^*\ni\alpha=\sum_{k=0}^N \alpha_{(k)}, 
  \label{eq:hngrad}
\end{equation}
where $\alpha_{(k)}\colon\HH_N\to\R$ is defined by $\alpha_{(k)}(x)\coloneq\alpha(x_k)$ with the notation \eqref{eq:hhx}.

Since $\HH_N$ is not a subalgebra of $\HH$, the notions of character and infinitesimal character on $\HH_N^*$ are not well-defined. We can however introduce their \emph{truncated} versions.
\begin{dfn}\label{def_trunc}
  We say that $X\in\HH_N^*\setminus\{0\}$ is a \emph{truncated character} on $\HH_N$ if
  \[ \langle X,xy\rangle=\langle X,x\rangle\langle X,y\rangle \]
  holds for all $x\in\HH_{(n)},y\in\HH_{(m)}$ with $n+m\le N$. We call $G^N$ the space of truncated characters on $\HH_N$.
  
  Likewise, we say that $\alpha\in\HH_{N}^*$ is a \emph{truncated infinitesimal character} if
  \[ \langle\alpha,xy\rangle=\langle\alpha,x\rangle\langle\varepsilon,y\rangle+\langle\varepsilon,x\rangle\langle\alpha,y\rangle \]
  holds for all $x\in\HH_{(n)},y\in\HH_{(m)}$ with $n+m\le N$. We call $\g^N$ the space of truncated
  infinitesimal characters on $\HH_N$.
\end{dfn}

\begin{lmm}\label{lmm:incl}
There are a canonical inclusions $\HH_N^*\hookrightarrow\HH_{N+1}^*\hookrightarrow\HH^*$, which induce canonical inclusion $\g^N\hookrightarrow\g^{N+1}\hookrightarrow\g$. Moreover such canonical inclusions are right-inverse for the
corresponding restriction maps $\HH^*\to\HH_{N+1}^*\to\HH_{N}^*$.
\end{lmm}
\begin{proof} Using the notation \eqref{eq:hngrad}, we can extend $\alpha\in\HH_N^*$ to $\alpha\in\HH_{N+1}^*$ (respectively $\HH^*$) by setting $\alpha_{(N+1)}\equiv 0$ (respectively $\alpha_{(k)}\equiv 0$ for all $k\geq N+1$).
  Trivially this extension takes $\HH_N^*$ to $\HH_{N+1}^*$. If $\alpha\in\g^N$ and $x,y\in\HH_N$ are such that $|x|+|y|\le N+1$ then
  \begin{align*}
    \langle\alpha,xy\rangle&=\Big\langle\alpha,\sum_{j=0}^{N+1}(xy)_{j}\Big\rangle=\sum_{j=0}^N\langle\alpha,(xy)_j\rangle
    = \sum_{j=0}^N\sum_{k=0}^j\langle\alpha,x_ky_{k-j}\rangle
    \\ & =\sum_{j=0}^N(\langle\alpha,x_j\rangle\langle\varepsilon,y\rangle+\langle\varepsilon,x\rangle\langle\alpha,y_j\rangle)
    = \langle\alpha,x\rangle\langle\varepsilon,y\rangle+\langle\varepsilon,x\rangle\langle\alpha,y\rangle.
  \end{align*}
  so that the extension of $\alpha$ is in $\g^{N+1}$.
  The same argument yields the inclusion $\g^N\hookrightarrow\g$.
\end{proof} 

There are also the truncated exponential $\exp_N\colon\HH_N^*\to\HH_N^*$ and logarithm $\log_N\colon\HH^*_N\to\HH^*_N$, defined by the sums
\begin{equation}
  \exp_N(\alpha)\coloneq\sum_{k=0}^N\frac{1}{k!}\left.\alpha^{\star k}\right|_{\HH_N},\qquad\log_N(X)\coloneq\sum_{k=1}^N\frac{(-1)^{k+1}}{k}\left.(X-\varepsilon)^{\star k}\right|_{\HH_N}.
  \label{eq:truncexp}
\end{equation}
The proof of the next result can be found for instance in \cite[Thm 77]{Foi}.
\begin{lmm}
  $(G^N,\star,\varepsilon)$ is a group and $(\g^N,[\cdot,\cdot])$ is a Lie algebra.
  Moreover, $\exp_N\colon\g^N\to G^N$ is a bijection with inverse $\log_N\colon G^N\to\g^N$.
\end{lmm}
For every $k\geq 0$ we define now, using the notation \eqref{eq:hngrad},
\[
  W_k\coloneq\left\{\alpha\in\g: \ \alpha=\alpha_{(k)}\right\}.
\]
\begin{lmm}\label{lmm:gradedLie}
For all $n,m\geq 0$ we have $[W_n,W_m]\subset W_{n+m}$.
\end{lmm}
\begin{proof}
Let $x\in\HH$. With the notation \eqref{eq:hhx} we have for $\alpha\in W_n$ and $\beta\in W_m$
\[
(\alpha\star\beta-\beta\star\alpha)(x)=(\alpha\otimes\beta-\beta\otimes\alpha)\Delta x=
(\alpha\otimes\beta-\beta\otimes\alpha)\Delta x_{n+m}
\]
by \eqref{eq:coprod}.
\end{proof}

By the canonical inclusion of Lemma \ref{lmm:incl}, we observe that 
\begin{equation}\label{WW} \g^N=\bigoplus_{k=1}^NW_k, \qquad \g^N\ni\alpha= \sum_{k=0}^N \alpha_{(k)}
\end{equation}
in the notation \eqref{eq:hngrad}.
With this decomposition $\g^N$ becomes by Lemma \ref{lmm:gradedLie} a graded Lie algebra.
We recall that the \emph{center} of $\g^N$ is the subspace of all $w\in\g^N$ such that $[\alpha,w]=0$ for all $\alpha\in\g^N$,
while the \emph{center} of $G^N$ is the set of all $X\in G^N$ such that $X\star Y=Y\star X$ for all $Y\in G^N$.

\begin{prp}\label{pr:centre}
  $W_N$ is contained in the center of $\g^N$ and $\exp_N(W_N)$ is a subgroup contained in the center of $G^N$. 
  \end{prp}
\begin{proof}
  Let $\alpha\in\g^N$ and $w\in W_N$. Clearly, $\langle[\alpha,w],x\rangle$ is zero unless $|x|=N$.
  In this case
  \[ \langle[\alpha,w],x\rangle=\langle\alpha\otimes w-w\otimes\alpha,\Delta x\rangle=\langle\alpha,1\rangle\langle w,x\rangle-\langle w,x\rangle\langle\alpha,1\rangle=0\]
  since $\langle w,y\rangle=\langle w,y_N\rangle$, in the notation \eqref{eq:hhx}. The second assertion follows easily: it is enough to write $X=\exp_N(w)$ and $Y=\exp_N(\alpha)$ with $\alpha\in\g^N$ and $w\in W_N$ and use the explicit representation \eqref{eq:truncexp} of $\exp_N$ and the fact that $\alpha\star w=w\star\alpha$.
\end{proof}

The next (famous) result describes the group law on $G^N$ in terms of an operation on $\g^N$ via the exponential/logarithmic map.
 
\begin{thm}[Baker--Campbell--Hausdorff]
  For all $\alpha,\beta\in\g^N$, we have \[ \log_N(\exp_N(\alpha)\star\exp_N(\beta))\in\g^N. \]
  \label{thlBCH}
\end{thm}
We define the map $\BCH_N\colon\g^N\times\g^N\to\g^N$ by
\begin{equation}
  \BCH_N(\alpha,\beta)\coloneq \log_N(\exp_N(\alpha)\star\exp_N(\beta)).
  \label{bch}
\end{equation}
Another way to interpret this theorem is to say that there exists an element $\gamma=\BCH_N(\alpha,\beta)\in\g_N$ such that $\exp_N(\alpha)\star\exp_N(\beta)=\exp_N(\gamma)$.

It is a classical result that the map $\BCH_N$ is formed by a sum of iterated Lie brackets of $\alpha$ and $\beta$, where the first terms are
\begin{equation}
  \BCH_N(\alpha,\beta)=\alpha+\beta+\frac{1}{2}[\alpha,\beta]+\frac{1}{12}[\alpha,[\alpha,\beta]]-\frac{1}{12}[\beta,[\alpha,\beta]]+\dotsb,
  \label{eq:DBCH}
\end{equation}
and the following ones are explicit but difficult to compute.
Nevertheless, fully explicit formulas have been known since 1947 by Dynkin \cite{Dynkin2000}.

For our purposes, however, Dynkin's formula is too complicated (for example, the regularity argument in step 2 of the proof of \Cref{thm:charext} would not be as evident) so we rely on a different expression first shown by Reutenauer \cite{Reutenauer1986}.
In order to describe it, let $\varphi_k\colon(\HH^*)^{\otimes k}\to \HH^*$ be the linear map
\begin{equation}
  \varphi_k(\alpha_1\otimes\dotsm\otimes \alpha_k)=\sum_{\sigma\in S_k}a_{\sigma}\, \alpha_{\sigma(1)}\star\dotsm\star \alpha_{\sigma(k)}
  \label{eq:phikdef}
\end{equation}
where $S_k$ denotes the symmetric group of order $k$, and $a_\sigma\coloneq\tfrac{(-1)^{d(\sigma)}}{k}\binom{k-1}{d(\sigma)}^{-1}$ is a constant depending only on the \emph{descent number} $d(\sigma)$ of the permutation $\sigma\in S_k$, namely the number of $i\in\{1,\ldots,k-1\}$ such that $\sigma(i)>\sigma(i+1)$.
\begin{lmm}[Reutenauer's formula]
  \label{lmm:reu}
For all $\alpha,\beta\in\g^N$
  \begin{equation}
    \BCH_N(\alpha,\beta)=\sum_{k=1}^N\sum_{i+j=k}\frac{1}{i!j!}\, \varphi_k(\alpha^{\otimes i}\otimes\beta^{\otimes j}).
    \label{eq:BCHk}
  \end{equation}
  Moreover, for all $i\in\{0,\ldots,N\}$, we have $\varphi_{N}\left(\alpha^{\otimes i}\otimes\beta^{\otimes (N-i)}\right)\in W_N$. 
\end{lmm}
\begin{proof}
  Let us suppose first that $T(V)$ is the (completed) tensor algebra over a two-dimensional vector space $V$, 
  with $V$ linearly generated by $\{e_1,e_2\}$. Then the result is contained in Reutenauer's paper \cite{Reutenauer1986} where the free step-N nilpotent Lie algebra $\mathfrak L^N$ plays the rôle of $\g^N$.
We want now to show how this implies the same result in our more general setting.

Let $\alpha,\beta\in\g^N$ and let $\Phi\colon(T(V)_N,\otimes)\to(\HH_N^*,\star)$ be the unique algebra morphism such that 
$\Phi(e_1)=\alpha$, $\Phi(e_2)=\beta$. Then $\Phi$ restricts to a Lie-algebra morphism $\Phi\colon\f L^N\to
\g^N$ such that $\BCH_N(\alpha,\beta)=\Phi(\BCH_N(e_1,e_2))$ and therefore \eqref{eq:BCHk} follows.

In order to prove the first formula, we first note that $\Phi$ is not a graded morphism, since the generators $e_1$ and $e_2$ are homogeneous of degree $1$ in $T(V)_N$, but $\alpha$ and $\beta$ are in general not homogeneous in $\HH_N^*$.
However, from the bilinearity of the Lie bracket and \Cref{lmm:gradedLie} we obtain 
\[
  [W_n\oplus\dotsb\oplus W_N,W_m\oplus\dotsb\oplus W_N]\subset W_{n+m}\oplus W_{n+m+1}\oplus\dotsb\oplus W_N.
  \]
Then, if $\alpha_1,\dotsc,\alpha_k\in\g^N$ then $\varphi_k(\alpha_1\otimes\dotsm\otimes\alpha_k)\in W_k\oplus\dotsb\oplus W_N$.
\end{proof}

From all these considerations we obtain the following result on the map 
\begin{equation}\label{eq:bch}
\BCH_{(n+1)}\colon \g^{n+1}\times\g^{n+1}\to W_{n+1}, \qquad \BCH_{(n+1)}\coloneq\BCH_{n+1}-\BCH_{n},
\qquad n\geq 0.
\end{equation}
Note that $\BCH_{(n+1)}$ takes indeed values in $W_{n+1}$ rather than in $\g^{n+1}$ by (both assertions of)
Lemma \ref{lmm:reu}.
\begin{lmm}
  Let $x\in\HH_{(n+1)}$ and $\alpha,\beta\in\g^{n+1}$. Then 
  \begin{equation} \label{eq:phik}
  \langle\BCH_{(n+1)}(\alpha,\beta),x\rangle=\sum_{i+j=n+1}\frac{1}{i!j!}\, \sum_{(x)}\sum_{\sigma\in S_{n+1}}a_\sigma\prod_{p=1}^i\langle\alpha,x_{(\sigma^{-1}(p))}\rangle\prod_{q=i+1}^{n+1}\langle\beta,x_{(\sigma^{-1}(q))}\rangle,
\end{equation}
where
  \[ \Delta'_{n}x = \sum_{(x)} x_{(1)}\otimes\dotsm\otimes x_{(n+1)}\in\HH_{(1)}^{\otimes (n+1)}. \]
  \label{lmm:phik}
\end{lmm}
\begin{proof}
Set $\alpha_1=\cdots=\alpha_i\coloneq\alpha$, $\alpha_{i+1}=\cdots=\alpha_{n+1}\coloneq\beta$.
Then the result follows directly from the definition of $\varphi_k$ in \eqref{eq:phikdef} together with \eqref{eq:kfoldconv} and the fact that since $\langle\alpha_j, \1\rangle=0$ we can write
  \begin{equation}\label{eq:starsta}
  \alpha_1\star\dotsm\star\alpha_{n+1} =
  (\alpha_1\otimes\dotsm\otimes\alpha_{n+1})\Delta'_{n} 
  \end{equation}
  instead (note the reduced coproduct in place of the full coproduct).
\end{proof}

\subsection{A distance on the group of truncated characters}
\label{sss:homnorm}
Now we introduce a distance on $G^N$ which is well adapted to the notion of rough paths, to be introduced in \Cref{dfn:genRP} below.
We fix a basis $B$ of $\HH_N$ and define a norm $\|\cdot\|$ on this space by requiring that $B$ is orthonormal. 
There is a unique function $c\colon B\times B\times B\to \R$ such that
\[
\Delta v=\sum_{v_1,v_2\in B} c(v,v_1,v_2)\, v_1\otimes v_2, \qquad \forall \, v\in B.
\]
Then we define
\[ K\coloneq \max_{v\in B}\sum_{v_1,v_2\in B}|c(v,v_1,v_2)|<\infty, \qquad
\lrn{f}\coloneq K \, \sup_{v\in B}|\langle f,v\rangle|, \quad f\in\HH_N^*. \]
Then, if $f,g\in\HH_{N}^*$, for any $v\in B$
\[ |\langle f\star g,v\rangle|\le \sum_{v_1,v_2\in B}|c(v,v_1,v_2)||\langle f,v_1\rangle||\langle g,v_2\rangle|\le\frac{1}{K}\, \lrn{f}\lrn{g}, \]
thus $\lrn{f\star g}\le\lrn{f}\lrn{g}$.
We set now for all $X\in G^N$
\begin{equation}
  |X| \coloneq \max_{k=1,\dotsc,N}\left( k!\lrn{X_{(k)}} \right)^{1/k}+\max_{k=1,\dotsc,N}\left( k!\lrn{\left(X^{-1}\right)_{(k)}} \right)^{1/k},
  \label{eq:homnorm}
\end{equation}
where for $X\in G^N\subset\HH^*_N$ we use the notation \eqref{eq:hngrad}. We define $G^N\times G^N\ni(X,Y)\mapsto \rho^N(X,Y)\coloneq|X^{-1}\star Y|\in\R_+$, i.e. by \eqref{eq:homnorm} 
\begin{equation}\label{eq:rho}
  \rho_N(X,Y) = \max_{k=1,\dotsc,N}\left( k!\lrn{(Y^{-1}\star X)_{(k)}} \right)^{1/k}+\max_{k=1,\dotsc,N}\left( k!\lrn{(X^{-1}\star Y)_{(k)}} \right)^{1/k}
\end{equation}

\begin{prp}
The map $\rho_N$ defines a left-invariant distance on the group $G^N$ such that the metric space $(G^N,\rho_N)$ is complete.
  \label{prp:subadd}
\end{prp}
\begin{proof}
  We only need to prove that the function $|\cdot|$ defined in \eqref{eq:homnorm} is sub-additive, the other properties being clear.
Note that for $X,Y\in G^N$, with the notation \eqref{eq:hngrad} we have
   \begin{equation} X\star Y = \left(\sum_{k=0}^NX_{(k)}\right)\star
  \left(\sum_{k=0}^NY_{(k)}\right)=\sum_{k=0}^N\sum_{j=0}^kX_{(j)}\star Y_{(k-j)}.
  \label{lmm:convcomp}
  \end{equation}
Therefore
  \begin{align*}
    \lrn{(X\star Y)_{(k)}} &\leq \sum_{j=0}^k\lrn{X_{(j)}}\lrn{Y_{(k-j)}}
    \le \frac{1}{k!}\sum_{j=0}^k\binom{k}{j}|X|^j|Y|^{k-j}
    = \frac{1}{k!}(|X|+|Y|)^{k}
  \end{align*}
  whence the result.
\end{proof}
The next result is the analog of \cite[Prop. 7]{Lyons2007}.
\begin{lmm}\label{lmm:LV7} If $X=\exp_N(w_1+\cdots+w_N)$ with $w_i\in W_i$, then
\[
c_N  \, \max_{k=1,\ldots,N} \lrn{w_k}^{1/k} \leq |X|\leq C_N \, \max_{k=1,\ldots,N} \lrn{w_k}^{1/k}.
\]
\end{lmm}
\begin{proof}
Using the notation \eqref{eq:hngrad}, we have
\[
X_{(k)}=\sum_{i=1}^k \frac1{i!} \sum_{j_1+\cdots+j_i=k} w_{j_1}\star\cdots\star w_{j_i}
\]
so that for all $k=1,\ldots,N$
\[\begin{split}
\left( k! \lrn{X_{(k)}}\right)^{1/k} 
& \leq \left(\sum_{i=1}^k\frac{k!}{i!}\sum_{j_1+\dotsb+j_i=k}\left(\lrn{w_{j_1}}^{1/j_1}\right)^{j_1}\dotsm\left(\lrn{w_{j_i}}^{1/j_i}\right)^{j_i}\right)^{1/k} 
\\ & \leq \left( \sum_{i=1}^k\frac{k!}{i!}\sum_{j_1+\dotsb+j_i=k}\left(\max_{\ell=1,\ldots,k}\lrn{w_\ell}^{1/\ell}\right)^{j_1+\dotsb+j_i} \right)^{1/k}.
\end{split}
\]
There are exactly $\binom{k-1}{i-1}\le\tfrac{(k-1)^{i-1}}{(i-1)!}$ solutions to $j_1+\dotsb+j_i=k$ so that
\[ (k!\lrn{X_{(k)}})^{1/k}\le(k!(\mathrm e^k-1))^{1/k}\max_{\ell=1,\ldots,k}\lrn{w_\ell}^{1/\ell}.\]
Since $X^{-1}=\exp_N(-w_1-\cdots-w_N)$, the bound for $X^{-1}$ follows in the same way and we have therefore
proved the desired upper bound for $|X|$.
For the lower bound, we use the truncated logarithm 
\[
w_k=\sum_{i=1}^k \frac{(-1)^{i-1}}{i} \sum_{j_1+\cdots+j_i=k} X_{(j_1)}\star\cdots\star X_{(j_i)}.
\]
Then we can estimate
\[\begin{split}
\lrn{w_k}^{1/k} 
& \leq \left(\sum_{i=1}^k\frac{1}{i}\sum_{j_1+\dotsb+j_i=k}\left(\lrn{X_{(j_1)}}^{1/j_1}\right)^{j_1}\dotsm\left(\lrn{X_{(j_i)}}^{1/j_i}\right)^{j_i}\right)^{1/k} 
\\ & \leq \left( \sum_{i=1}^k\frac{1}{i}\binom{k-1}{i-1}\left(\max_{\ell=1,\ldots,k}\lrn{X_\ell}^{1/\ell}\right)^{j_1+\dotsb+j_i} \right)^{1/k}
\leq \frac1{c_N} |X|
\end{split}\]
and the proof is complete.
\end{proof}

We now note that the function $|\cdot|$ and the distance $\rho^N$ make $G^N$ a \emph{homogeneous group}, see \cite{Folland1982} for an
extensive treatment of this subject, and \cite{Lyons2007} for the case of tensor algebras and the relation with geometric rough paths.  

To put it briefly, for all $r>0$ we can define the following linear operator $\Omega_r\colon\HH^*\to\HH^*$
\[
\Omega_r\alpha\coloneq \sum_{k\geq 0} r^k\, \alpha_{(k)}.
\]
This family satisfies $\Omega_r\circ\Omega_s=\Omega_{rs}$, $r,s>0$. Moreover 
$\Omega_r\colon\g^N\to\g^N$ is a Lie-algebra automorphism of $\g^N$ for all $r>0$.
Then they induce group automorphisms $\Lambda_r\coloneq{\exp}_N\circ\Omega_r\circ{\log}_N\colon G^N\to G^N$, $r>0$. 
In the terminology of \cite{Folland1982}, $(\Omega_r)_{r>0}$ is a family of \emph{dilations}
on the finite-dimensional Lie algebra $\g^N$ and $G^N$
is a \emph{homogeneous} group. 

Note that the function $|\cdot|\colon G^N\to\R_+$ is continuous, satisfies $|\Lambda_rX|=r|X|$ for all $r>0$ and $X\in G^N$,
and $|X|=0$ for $X\in G^N$ if and only if $X=\1$. These three properties make $|\cdot|$ a \emph{homogeneous norm} on $G^N$, see \cite{Folland1982}.
The homogeneity property plays an important role in the proof of \Cref{thm:charext} below.

\section{Construction of Rough paths}
\label{sse:RPthm}
As in the previous section, we fix a locally-finite graded connected Hopf algebra $\HH$.
We also fix a number $\gamma\in\,]0,1[$ and let $N\coloneq\lfloor\gamma^{-1}\rfloor$ be the biggest integer such that $N\gamma\le 1$.
  Without loss of generality we can fix a basis $B$ of $\HH_N$ consisting only of homogeneous elements and in particular we let $\{e_1,\dotsc,e_d\}=B\cap\HH_{(1)}$ where $d\coloneq\dim\HH_{(1)}$.

\begin{dfn}
  A \emph{$(\HH,\gamma)$-rough path} is a function $X\colon[0,1]^2\to G^N$, with $N=\lfloor\gamma^{-1}\rfloor$, which satisfies Chen's rule 
  \begin{equation}\label{eq:chen}
  X_{su}\star X_{ut}=X_{st}, \qquad s,u,t\in[0,1],
  \end{equation}
and such that for all $v\in B$
  \begin{equation}
    |\langle X_{st},v\rangle|\lesssim|t-s|^{\gamma|v|}.
    \label{eq:genrpbound}
  \end{equation}
  If $x^i\colon[0,1]\to\R$, $i=1,\ldots,d$, is such that 
  $x^i_t-x^i_s=\langle X_{st},e_i\rangle$, $s,t\in[0,1]$, we say that $X$ is a $\gamma$-rough path over $(x^1,\dotsc,x^d)$.
  \label{dfn:genRP}
\end{dfn}

\begin{rmk}
  By specializing this definition to different choices of $\HH$ we recover both \emph{geometric rough paths} \cite{Lyons1998} where $\HH$ is the shuffle Hopf algebra over an alphabet, \emph{branched rough paths} \cite{Gubinelli2010} where $\HH$ is the Butcher--Connes--Kreimer Hopf algebra on decorated non-planar rooted trees, and also \emph{planarly branched rough paths} \cite{Curry2018}.
\end{rmk}

We remark that there is a bijection between
\begin{enumerate}
\item functions $X\colon[0,1]^2\to G^N$ such that $X_{su}\star X_{ut}=X_{st}$,  for all $s,u,t\in[0,1]$,
\item functions $\bb X \colon[0,1]\to G^N$ such that $\bb X_0=\1$,
\end{enumerate}
given by 
\begin{equation}\label{eq:XbbX}
X\mapsto \bb X, \quad \bb X_t\coloneq X_{0t}, \qquad \bb X\mapsto X, \quad X_{st}\coloneq\bb X_s^{-1}\star\bb X_t, \qquad s,t\in[0,1].
\end{equation}

\begin{prp}  \label{prp:genrphold}
  Let $\bb X\colon[0,1]\to G^N$ and $X\colon[0,1]^2\to G^N$ as in \eqref{eq:XbbX}.
  Then $X$ is a $(\HH,\gamma)$-rough path as in \Cref{dfn:genRP} if and only if $\bb X$ is $\gamma$-Hölder with respect to the metric $\rho_N$ defined in \eqref{eq:homnorm}.
\end{prp}
\begin{proof}
  First note that the distance in \eqref{eq:rho} is defined with respect to a fixed (but arbitrary) basis so we use the basis $B$ fixed at the beginning of this section.
  Also, due to the above remark we only have to verify that $\bb X$ is $\gamma$-Hölder with respect to $\rho_N$ if and only if $X$ satisfies \eqref{eq:genrpbound} using the same basis.
  In one direction, if $\bb X$ is $\gamma$-Hölder then, by definition
  \[ |X_{st}|=\rho_N(\bb X_s,\bb X_t)\lesssim|t-s|^\gamma \]
  and so, for a basis element $v\in B$ we have
  \[ |\langle X_{st},v\rangle|\lesssim |t-s|^{\gamma|v|}. \]
  Conversely, if \eqref{eq:genrpbound} holds then $|X_{st}|\lesssim|t-s|^{\gamma}$ and so by definition also $\rho_N(\bb X_s,\bb X_t)\lesssim|t-s|^\gamma$, i.e. $\bb X$ is $\gamma$-Hölder with respect to $\rho_N$.
\end{proof}

We now come to the problem of existence.
Our construction of a rough path in the sense of \Cref{dfn:genRP} over an arbitrary collection of $\gamma$-Hölder paths $(x^1,\dotsc,x^d)$ relies in the following extension theorem. We note that the
proof is a reinterpretation of the approach of Lyons-Victoir \cite[Theorem 1]{Lyons2007} in the context of a more general graded Hopf-algebra $\HH$.
\begin{thm}[Rough path extension]
Let $1\le n\le N-1$ and $\gamma\in\,]0,1[$ such that $\gamma^{-1}\not\in\N$. Suppose we have a $\gamma$-H\"older path $\bb X^n\colon[0,1]\to(G^n,\rho_n)$.
  There is a $\gamma$-H\"older path ${\bb X}^{n+1}\colon[0,1]\to(G^{n+1},\rho_{n+1})$ extending $\bb X^n$, i.e. such that $\left.{\bb X}^{n+1}\right|_{\HH_n}=\bb X^n$.
  \label{thm:charext}
\end{thm}
A key tool is the following technical lemma whose proof can be found in \cite[Lemma 2]{Lyons2007}.
\begin{lmm}\label{lmm:LV}
  Let $(E,\rho)$ be a complete metric space and set \[D=\{t_k^m\coloneq k2^{-m}:m\ge 0,k=0,\dotsc,2^m-1\}.\]
  Suppose $y\colon D\to E$ is a path satisfying the bound $\rho(y_{t^m_k},y_{t^m_{k+1}})\lesssim 2^{-\gamma m}$ for some $\gamma\in(0,1)$.
  Then, there exists a $\gamma$-Hölder path $x\colon[0,1]\to E$ such that $\left.x\right|_D=y$.
\end{lmm}

\begin{proof}[Proof of \Cref{thm:charext}]
  The construction of ${\bb X}^{n+1}$ is made in two steps.
  \paragraph{\bfseries Step 1.}
  For $m\geq 0$ and $k\in\{0,\ldots,2^m\}$ we define $ t^m_k\coloneq k2^{-m}\in[0,1]$. Then we define the following sets of dyadics in $[0,1]$
  \[
D_{(m)}\coloneq\{t^m_k \mid k=0,\dotsc,2^m\}, \qquad D_m\coloneq\bigcup_{n=0,\ldots,m} D_{(n)},
\qquad D\coloneq\bigcup_{m\geq 0} D_{(m)}.
  \]
  Set $X_{st}=(\bb X^n_s)^{-1}\star\bb X_t^n\in G^n$ and $L_{st} = \log_n(X_{st})\in\g^n$ where $\log_n$ was defined in \eqref{eq:truncexp}.
  Then, the Baker--Campbell--Hausdorff formula \eqref{bch} and Chen's rule \eqref{eq:chen} imply that
  \begin{equation}
   L_{st} = \BCH_n(L_{su},L_{ut}). 
  \label{eq:BCHL}
  \end{equation}
  
  We look for $Y\colon[0,1]^2\to G^{n+1}$ such that $Y$ satisfies Chen's rule \eqref{eq:chen} and 
  $Y\big|_{\HH_N}=X$. We use 
  throughout the proof that $\g^n\subset\g^{n+1}$, see \Cref{lmm:incl}. 
  
  In a first step, we define $Y\colon D\times D\to G^{n+1}$. In the second step we show that $Y$ has suitable uniform
  continuity properties and can thus be extended to $[0,1]^2$ using \Cref{lmm:LV}.
  
  The construction of $Y\colon D\times D\to G^{n+1}$ goes through a construction of $Y^m\colon D_m\times D_m\to G^{n+1}$ by recursion on $m\geq 0$.
We claim that for all $m\geq 0$ we can find $Y^m$ such that
\begin{enumerate}
\item $Y^m$ satisfies Chen's relation on $D_m$, namely 
$Y_{ab}^m\star Y_{bc}^m=Y^m_{ac}$ for all $a,b,c\in D_m$
\item for any $n\in\{0,\ldots,m\}$ and $k,\ell\in\{0,\ldots,2^{m-n}\}$, we have the compatibility relation
\[
Y^m_{t^m_{k2^n}t^m_{\ell 2^n}}=Y^{m-n}_{t^{m-n}_kt^{m-n}_\ell}.
\]
\item $Y^m$ restricted to $\HH_n$ is equal to $X\colon D_m\times D_m\to G^n$, in the sense that
\[
\left.Y^m_{ab}\right|_{\HH_n} = X_{ab}, \qquad \forall \, a,b\in D_m.
\]
\item for all $k=0,\ldots,2^m-1$, setting 
$$Z^m_{t^m_kt^m_{k+1}}\coloneq\log_{n+1}\left(Y^m_{t^m_kt^m_{k+1}}\star \exp_{n+1}\left(-L_{t^m_kt^m_{k+1}}\right)\right),$$
we have $Z^m_{t^m_kt^m_{k+1}}\in W_{n+1}$.
\end{enumerate}

For $m=0$, we set 
$Y_{01}^0=\exp_{n+1}(L_{01})$, $Y_{00}^0=Y_{11}^0\coloneq\varepsilon$, and $Z^0_{01}\coloneq0\in W_{n+1}$. For $x\in\HH_n$,
we have $\langle\exp_{n+1}(L_{01}),x\rangle=\langle\exp_{n}(L_{01}),x\rangle$, so that $Y^0$ restricted to $\HH_n$ is equal to $X\colon D_0\times D_0\to G^n$.

Let now $m\geq 1$, and suppose that $Y^{m-1}\colon D_{m-1}\times D_{m-1}\to G^{n+1}$ has been constructed with the above properties.
We start by defining $Y^m_{tt}=\varepsilon$ for all $t\in D_{(m)}$.
Let us consider three consecutive points in $D_{(m)}$ of the form 
$$s=t^m_{2k}, \qquad u=t^m_{2k+1}, \qquad t=t^m_{2k+2}$$ 
for some $k=0,\ldots, 2^{m-1}-1$. Note that $s=t^{m-1}_k$ and $t=t^{m-1}_{k+1}$, so that $Z^m_{st}\coloneq Z^{m-1}_{st}\in W_{n+1}$ is already defined by the recurrence hypothesis. We define $Z^m_{su}$ and $Z^m_{ut}$ as follows
 \begin{equation} \label{eq:Zdef0}    
    Z^m_{su}=Z^m_{ut}\coloneq\frac{1}{2}\big(Z^{m-1}_{st}-\BCH_{(n+1)}(L_{su},L_{ut})\big), 
     \end{equation}
where $\BCH_{(n+1)}=\BCH_{n+1}-\BCH_n\colon\g^{n+1}\times\g^{n+1}\to W_{n+1}$, see \eqref{eq:bch}. Since by recurrence $Z^{m-1}_{st}\in W_{n+1}$, we obtain that $Z^m_{su},Z^m_{ut}\in W_{n+1}$ and 
 \begin{equation}
    \label{eq:Zdef}
 Z^m_{su}+Z^m_{ut}=Z^{m-1}_{st}-\BCH_{(n+1)}(L_{su},L_{ut}) = L_{st}+Z^{m}_{st} -\BCH_{n+1}(L_{su},L_{ut})
  \end{equation}
  where in the last equality we have applied \eqref{eq:BCHL}. Then we set 
  \[
  Y^m_{su}\coloneq\exp_{n+1}(L_{su}+Z^m_{su}), \qquad
  Y^m_{ut}\coloneq\exp_{n+1}(L_{ut}+Z^m_{ut}).
  \]
Since $\exp_{n+1}(W_{n+1})$ is in the center of $G^{n+1}$ by \Cref{pr:centre}, we obtain that
  \[ Y^m_{su} =\exp_{n+1}(L_{su})\star\exp_{n+1}(Z^m_{su}),\qquad Y^m_{ut}=\exp_{n+1}(L_{ut})\star
  \exp_{n+1}(Z^m_{ut}). \]
  By \eqref{bch} and \eqref{eq:Zdef} the product is equal to
  \[
  Y^m_{su}\star Y^m_{ut} = \exp_{n+1}(\BCH_{n+1}(L_{su},L_{ut})+Z^m_{su}+Z^m_{ut})
  =\exp_{n+1}(L_{st} + Z^m_{st})=Y^m_{st}.
  \]
  Let now $t^m_j,t^m_k\in D_{(m)}$ with $0\leq j<k\leq 2^m$.
We set
  \[ Y^m_{t^m_jt^m_k}\coloneq Y^m_{t^m_jt^m_{j+1}}\star \dotsm\star Y^m_{t^m_{k-1}t^m_k}, 
  \qquad Y^m_{t^m_kt^m_j}\coloneq \left(Y^m_{t^m_jt^m_k}\right)^{-1} \]
  so that the identity $Y^m_{ab}\star Y^m_{bc}=Y^m_{ac}$ is valid for any $a,b,c\in D_{(m)}$.
  
We need now to check that this definition is compatible with the values already constructed on $D_{m-1}\times D_{m-1}$. By the recursion assumption, it is enough to show that for all $k,\ell\in\{0,\ldots,2^{m-1}\}$
\[
Y^m_{t^m_{2k}t^m_{2\ell}}=Y^{m-1}_{t^{m-1}_kt^{m-1}_\ell}.
\]
If $k=\ell$ or $|k-\ell|=1$, then this is true by construction. Otherwise, if for example $k+1<\ell$ then
\[
Y^m_{t^m_{2k}t^m_{2\ell}} = Y^m_{t^m_{2k}t^m_{2k+2}}\star\cdots\star Y^m_{t^m_{2\ell-2}t^m_{2\ell}}
= Y^{m-1}_{t^{m-1}_{k}t^{m-1}_{k+1}}\star\cdots\star Y^{m-1}_{t^{m-1}_{\ell-1}t^{m-1}_{\ell}}
=Y^{m-1}_{t^{m-1}_kt^{m-1}_\ell}
\]
by the recursion property and the Chen relation satisfied by $Y^m$ (respectively $Y^{m-1}$) on $D_{(m)}$
(resp. $D_{(m-1)}$). 

We also have to check the extension property: for $x\in\HH_n$ we have
\[
\langle Y^m_{t^m_jt^m_{j+1}},x\rangle = \langle \exp_{n+1}(L_{t^m_jt^m_{j+1}})\star
\exp_{n+1}(Z^m_{t^m_jt^m_{j+1}}),x\rangle
=\langle \exp_{n}(L_{t^m_jt^m_{j+1}}),x\rangle = \langle X_{t^m_jt^m_{j+1}},x\rangle.
\]

By recurrence, we have proved that $Y^m\colon D_m\times D_m\to G^{n+1}$ is well defined for all $m\geq 0$,
with the above properties. 
Therefore, we can unambiguously define $Y\colon D\times D\to G^{n+1}$,
\[
Y_{st}\coloneq Y^m_{st}, \qquad s,t\in D_m,
\]
and $Y$ indeed satisfies the Chen relation on $D$, namely $Y_{ab}\star Y_{bc}=Y_{ac}$ for all $a,b,c\in D$,
and the restriction property
\[
\langle Y_{ab},x\rangle = \langle X_{ab},x\rangle, \qquad \forall \, a,b\in D, \ x\in\HH_n.
\]

  \paragraph{\bfseries Step 2.}
  In order to have a $(\HH_{n+1},\gamma)$-H\"older path, \Cref{dfn:genRP} requires us to construct a $\gamma$-Hölder path with values in $G^{n+1}$, and for this we will use \Cref{lmm:LV}.
  Set \[ a_m\coloneq 2^{m(n+1)\gamma}\max_{k=0,\dotsc,2^m-1}\lrn{Z^m_{t^m_kt^m_{k+1}}}_{n+1}. \]
  Then, if $\upsilon$ is a basis element in $\HH_{(n+1)}$ we have by \eqref{eq:phik}, for $s=t^m_k$, $u=t^m_{k+1}$ and $t=t^m_{k+2}$
  \begin{equation*}
    |\langle\BCH_{(n+1)}(L_{su},L_{ut}),\upsilon\rangle| \!\le \!\sum_{(\upsilon)}\!\sum_{i+j=n+1}\frac{1}{i!j!}\!\sum_{\sigma\in S_{n+1}}\!\!|a_\sigma|\prod_{p=1}^i|\langle L_{su},\upsilon_{(\sigma^{-1}(p))}\rangle|\prod_{q=i+1}^{n+1}\!\!|\langle L_{ut},\upsilon_{(\sigma^{-1}(q))}\rangle|.
  \end{equation*}
Now, since $\upsilon_{(j)}\in\HH_{(1)}$ for all $j=1,\dotsc,n+1$ we actually have that
  \[ |\langle L_{su},\upsilon_{(j)}\rangle|\le \sum_{k=1}^d|x_u^k-x_s^k||\upsilon_{(j)}^k|\le 2^{-m\gamma}\sum_{k=1}^d|\upsilon_{(j)}^k| \]
  for some coefficients $\upsilon_{(j)}^k\in\R$ such that $\upsilon_{(j)}=\sum_{k=1}^d\upsilon_{(j)}^k e_k$, and we have a similar estimate for $L_{ut}$ instead of $L_{su}$.
  Therefore we obtain that
  \[ \lrn{\BCH_{(n+1)}(L_{su},L_{ut})}_{n+1} \le C \,2^{-m(n+1)\gamma}, \]
  where
    \[ C = K\max_{\upsilon}\sum_{(\upsilon)}\sum_{i+j=n+1}\frac{1}{i!j!}\sum_{\sigma\in S_{n+1}}|a_\sigma|\sum_{k_1,\dotsc,k_{n+1}=1}^{n+1}\prod_{\ell=1}^{n+1}|\upsilon_{(\ell)}^{k_\ell}|.
  \]
  Therefore, from \eqref{eq:Zdef0} we get
  \[
  \max_{k=0,\dotsc,2^m-1}\lrn{Z^m_{t^m_kt^m_{k+1}}}_{n+1} \leq \frac12\max_{k=0,\dotsc,2^{m-1}-1}\lrn{Z^{m-1}_{t^{m-1}_k,t^{m-1}_{k+1}}}_{n+1} + \frac12\, C \,2^{-m(n+1)\gamma} 
  \]
  hence
  \begin{equation*}
    a_{m}\le 2^{(n+1)\gamma-1}a_{m-1}+\frac C2, \qquad m\geq 1.
  \end{equation*}
  Since $a_0=0$ we can show by recurrence on $m\geq 0$
  \[ a_m\le \frac C2\sum_{j=0}^{m-1}2^{-j(1-(n+1)\gamma)}. \]
  Since we are in the regime where $(n+1)\gamma<1$ (here we use that $\gamma^{-1}\notin\N$) we obtain that
  \[
   \sup_{m\ge 0}a_m \le\frac{C}{2-2^{(n+1)\gamma}}.
  \]
  Therefore
  \begin{equation}\label{eq:boundZ}
  \lrn{Z^m_{t^m_kt^m_{k+1}}}_{n+1} \lesssim 2^{-m(n+1)\gamma}, \qquad \forall \ {m\geq 0}, \ {k=0,\dotsc,2^m-1}.
  \end{equation}

  Let now fix $m\geq 0$, $i\in\{0,\ldots,2^m-1\}$, and set $s\coloneq t^m_j$, $t\coloneq t^m_{j+1}$. Then we want to prove that
   $\left|Y_{st}\right|\lesssim 2^{-m\gamma}$,
 see \eqref{eq:homnorm} for the definition of $|\cdot|$. By subadditivity of $|\cdot|$ w.r.t. the convolution
 product $\star$ we have
 \[
\begin{split}
\left|Y_{st}\right| & \leq \left|\exp_{n+1}(L_{st})\right|+\left|\exp_{n+1}(
Z^m_{st})\right|.
\end{split} 
 \]
 By \Cref{lmm:LV7} and \eqref{eq:boundZ}
 \[
\begin{split}
\left|\exp_{n+1}(Z^m_{st})\right|\lesssim \lrn{Z^m_{t^m_kt^m_{k+1}}}_{n+1}^{\frac1{n+1}} \lesssim 2^{-m\gamma}.
\end{split} 
 \]
 Moreover, using \Cref{lmm:LV7} again (first the upper bound, then the lower bound) and the fact that ${\bb X^n}\colon[0,1]\to G^n$ is $\gamma$-H\"older by assumption,
 \[
\begin{split}
\left|\exp_{n+1}(L_{st})\right|& \leq C_{n+1}\sup_{k=1,\ldots,n+1} \lrn{\left(L_{st}\right)_k}^{1/k}
= C_{n+1}\sup_{k=1,\ldots,n} \lrn{\left(L_{st}\right)_k}^{1/k} \leq 
\\ & \leq \frac{C_{n+1}}{c_n} \, \left|\exp_{n}(L_{st})\right| = \frac{C_{n+1}}{c_n} \, \left|X_{st}\right| 
=\frac{C_{n+1}}{c_n} \, \rho_{n}\left(\bb X^{n}_{t^m_j},\bb X^{n}_{t^m_{j+1}}\right)\lesssim 2^{-m\gamma}.
 \end{split}
 \]
   Therefore, the path $\bb X^{n+1}\colon D\to G^{n+1}$ defined by $\bb X^{n+1}_{t^m_j}\coloneq Y_{0,t^m_j}$ satisfies
  \[ 
  \rho_{n+1}\left(\bb X^{n+1}_{t^m_j},\bb X^{n+1}_{t^m_{j+1}}\right)\lesssim 2^{-m\gamma}, 
  \]
  thus by \Cref{lmm:LV} we obtain a $\gamma$-H\"older path ${\bb X}^{n+1}\colon[0,1]\to G^{n+1}$ extending 
  $\bb X^{n}$.
\end{proof}

\begin{rmk}\label{rmk:change}
  Our construction depends on a finite number of choices, namely we set $Z_{01}=0$ to start the recursion in \eqref{eq:Zdef}, and this for each level; moreover in \eqref{eq:Zdef} we make the choice $Z_{t^m_{2k},t^m_{2k+1}}=Z_{t^m_{2k+1},t^m_{2k+2}}$. These choices are the same as in \cite[Proof of Theorem 1]{Lyons2007} and are indeed the most natural ones, but one could change them and the
final outcome would be different.
\end{rmk}
\begin{rmk}
While in \cite[Proof of Proposition 6]{Lyons2007} Lyons and Victoir use the axiom of choice, our proof is completely constructive. In particular, we use the explicit map $\exp_{k+1}\circ\log_k \colon G^k(\T_n) \to G^{k+1}(\T_n)$ which plays the role
of the injection $i_{G/K,G}\colon G/K\to G$ in \cite[Proposition 6]{Lyons2007}. The fact that this map has good
continuity estimates is based on \Cref{lmm:LV7}.
\end{rmk}

\begin{crl}\label{crl:3.8}
  Given $\gamma\in\,]0,1[$ with $\gamma^{-1}\notin\N$ and a collection of $\gamma$-Hölder paths $x^i\colon[0,1]\to\R$, $i=1,\ldots,d$, there exists a $\gamma$-H\"older path $\bb X\colon[0,1]\to G^N$ such that $\langle \bb X,e_i\rangle =x^i-x^i_0$, $i=1,\ldots,d$. Then $X_{st}\coloneq\bb X_s^{-1}\star\bb X_t$
defines a $(\HH,\gamma)$-rough path over $(x^1,\dotsc,x^d)$.
\end{crl}
\begin{proof}
  We start with the following observation: for $n=1$, the group $G^1\subset\HH_{(1)}^*$ is abelian, and isomorphic to the additive group $\HH_{(1)}^*$.
  Indeed, let $X,Y\in G^1$ and $x\in\HH_{(1)}$. Then, as $\Delta x=x\otimes\1+\1\otimes x$ by the grading, we have that
  \begin{align*}
    \langle X\star Y,x\rangle = \langle X,x\rangle +\langle Y,x\rangle,
  \end{align*}
  that is, $X\star Y=X+Y$.
  Moreover, in $\HH_1$ the product $xy=0$.
  Therefore, we may set $\langle\bb X^1_t,e_i\rangle\coloneq x^i_t-x^i_0$ where $\{e_1,\dotsc,e_d\}$ is a basis of $\HH_{(1)}$ and this path is $\gamma$-Hölder with respect to $\rho_1$.

  By \Cref{thm:charext} there is a $\gamma$-Hölder path $\bb X^2\colon[0,1]\to(G^2,\rho_2)$ extending $\bb X^1$ so in particular $\langle\bb X^2_t,e_i\rangle=x^i_t-x^i_0$ also.
  Continuing in this way we obtain successive $\gamma$-Hölder extensions $\bb X^3,\dotsc,\bb X^N$ and we set $\bb X\coloneq\bb X^N$.
\end{proof}

The following result has already been proved in the case where the underlying Hopf algebra $\HH$ is \emph{combinatorial} by Curry, Ebrahimi-Fard, Manchon and Munthe-Kaas in \cite[Theorem 4.3]{Curry2018}.
We remark that their proof works without modifications in our context so we have

\begin{thm}
  Let $\bb X\colon[0,1]\to G^N$ be a $\gamma$-H\"older path with $\bb X_0=\1$ and suppose that
  $\gamma^{-1}\notin\N$. There exists a path $\hat{\bb X}\colon[0,1]\to G$ such that $|\langle\hat{\bb X}_s^{-1}\star\hat{\bb X}_t,v\rangle|\lesssim|t-s|^{\gamma|v|}$ for all homogeneous $v\in\HH$ and extending $\bb X$, in the sense that $\hat{\bb X}\big|_{\HH_N}=\bb X$.
  \label{thm:geomext}
\end{thm}
\begin{rmk}
  In view of \Cref{thm:geomext} we can replace the truncated group in \Cref{dfn:genRP} by the full group of characters $G$.
  What this means is that $\gamma$-rough paths are uniquely defined once we fix the first $N$ levels and since $\HH$ is locally finite, this amounts to a finite number of choices. This is of course a generalization of the extension theorem
  of \cite{Lyons1998}, see also \cite[Theorem 7.3]{Gubinelli2010} for the branched case.
\end{rmk}

\section{Applications}
\label{sse:bRP}
We now apply \Cref{thm:charext} to various kinds of Hopf algebras in order to link this result with the contexts already existing in the literature.

\subsection{Geometric rough paths}
\label{sss:grp}
In this setting we fix a finite alphabet $A\coloneq\{1,\dotsc,d\}$.
As a vector space $\HH\coloneq T(A)$ is the linear span of the \emph{free monoid} $\mathrm M(A)$ generated by $A$. The product on $ \HH$ is the shuffle product $\shuffle\colon \HH\otimes \HH\to \HH$
defined recursively by $\1\shuffle v=v\shuffle\1=v$ for all $v\in \HH$, where $\1\in\mathrm M(A)$ is the unit for the monoid operation, and
\[ (au\shuffle bv) = a(u\shuffle bv) + b(au\shuffle v) \]
for all $u,v\in \HH$ and $a,b\in A$, where $au$ and $bv$ denote the product of the letters $a,b$ with the words $u,v$ in $\mathrm M(A)$.

The coproduct $\bar\Delta\colon\HH\to \HH\otimes \HH$ is obtained by \emph{deconcatenation} of words,
\[ \bar\Delta(a_1\dotsm a_n) = a_1\dotsm a_n\otimes\1+\1\otimes a_1\dotsm a_n + \sum_{k=1}^{n-1}a_1\dotsm a_k\otimes a_{k+1}\dotsm a_n.\]
It turns out that $(\HH,\shuffle,\bar\Delta)$ is a commutative unital Hopf algebra, and $(\HH,\bar\Delta)$ is the cofree coalgebra over the linear span of $A$.
The antipode is the linear map $S\colon\HH\to \HH$ given by
\[ S(a_1\dotsm a_n) = (-1)^na_n\dots a_1. \]
Finally, we recall that $\HH$ is graded by the length $\ell(a_1\dotsm a_n) = n$ and it is also connected.
The homogeneous components $\HH_{(n)}$ are spanned by the sets $\{a_1\dotsm a_n:a_i\in A\}$.

\Cref{dfn:genRP} specializes in this case to \emph{geometric rough paths} (GRP) as defined in \cite{Hairer2014} (see just below for the precise definition) and \Cref{thm:charext} coincides with \cite[Theorem 6]{Lyons2007}.
\begin{dfn}\label{def:geo}
  Let $\gamma\in\,]0,1[$ and set $N\coloneq\lfloor\gamma^{-1}\rfloor$. A \emph{geometric $\gamma$-rough path} is a map $X\colon[0,1]^2\to G^N$ which satisfies Chen's rule
    \[ X_{st} = X_{su}\star X_{ut} \]
  for all $s,u,t\in[0,1]$ and the analytic bound $|\langle X_{st},v\rangle|\lesssim|t-s|^{\gamma\ell(v)}$ for all $v\in \HH_N$.
\end{dfn}
Then \Cref{prp:genrphold} and the existence results \Cref{thm:charext}-\Cref{crl:3.8} are the content of the paper \cite{Lyons2007} by Lyons and Victoir.

\subsection{Branched rough paths}
\label{sec:brp}
Let $\T$ be the collection of all non-planar non-empty rooted trees with nodes decorated by $\{1,\dotsc,d\}$.
Elements of $\T$ are written as $2$-tuples $\tau=(T,c)$ where $T$ is a non-planar tree with node set $N_T$ and edge set $E_T$, and $c\colon N_T\to\{1,\dotsc,d\}$ is a function.
Edges in $E_T$ are oriented away from the root, but this is not reflected in our graphical representation. Examples of elements of $\T$ include the following
\[ \dtR<i>,\quad\dtI<ij>,\quad\dtV<ijk>,\quad\Forest{[i[j[k][l]][m]]}. \]

For $\tau\in\T $ write $|\tau|=\#N_T$ for its number of nodes. Also, given an edge $e=(x,y)\in E_T$ we set $s(e)=x$ and $t(e)=y$.
There is a natural partial order relation on $N_T$ where $x\le y$ if and only if there is a path in $T$ from the root to $y$ containing $x$.

We denote by $\F$ the collection of decorated rooted forests and we let $\HH\coloneq\HH_{\rm BCK}$ denote the vector space spanned by $\F$.
There is a natural commutative and associative product on $\F$, denoted by $\cdot$ and given by the disjoint union of forests, where the empty forest $\1$ acts as the unit.
Then, $\HH$ is the free commutative algebra over $\T$, with grading $|\tau_1\dotsm\tau_k|=|\tau_1|+\dotsb+|\tau_k|$. 
Given $i\in\{1,\dotsc,d\}$ and a forest $\tau=\tau_1\dotsm\tau_k$ we denote by $[\tau_1\dotsm\tau_k]_i$ the tree obtained by grafting each of the trees $\tau_1,\dotsc,\tau_k$ to a new root decorated by $i$, e.g.
\[ [\dtR<j>]_i = \dtI<ij>,\quad[\dtR<j>\dtR<k>]_i = \dtV<ijk>. \]

The decorated Butcher--Connes--Kreimer coproduct \cite{Connes1998, Gubinelli2010} is the unique algebra morphism $\Delta\colon\HH\to\HH\otimes\HH$ such that \[ \Delta[\tau]_i = [\tau]_i\otimes\1 + ({\id}\otimes[\cdot]_i)\Delta\tau. \]
This coproduct admits a representation in terms of \emph{cuts}.
An \emph{admissible cut} $C$ of a tree $T$ is a non-empty subset of $E_T$ such that any path from any vertex of the tree to the root contains at most one edge from $C$; we denote by $\A(T)$ the set of all admissible cuts of the tree $T$.
Any admissible cut $C$ containing $k$ edges maps a tree $T$ to a forest $C(T)=T_1\dotsm T_{k+1}$ obtained by removing each of the edges in $C$.
Observe that only one of the remaining trees $T_1,\dotsc,T_{k+1}$ contains the root of $T$, which we denote by $R^C(T)$; the forest formed by the other $k$ factors is denoted by $P^C(T)$.
This naturally induces a map on decorated trees by considering cuts of the underlying tree, and restriction of the decoration map to each of the rooted subtrees $T_1,\dotsc,T_{k+1}$.
Then,
\begin{equation}
  \Delta\tau=\tau\otimes\1+\1\otimes\tau+\sum_{C\in\A(\tau)}P^C(\tau)\otimes R^C(\tau).
  \label{eq:CKcuts}
\end{equation}
This, together with the counit map $\varepsilon\colon\F\to\R$ such that $\varepsilon(\tau)=1$ if and only if $\tau=\1$ endows $\F$ with a connected graded commutative non-cocommutative bialgebra structure, hence a Hopf algebra structure \cite{Manchon2008}.

As before we denote by $\HH ^*$ the linear dual of $\HH $ which is an algebra via the convolution product $\langle X\star Y,\tau\rangle=\langle X\otimes Y,\Delta\tau\rangle$ and we denote by $G$ the set of characters on $\HH $, that is, linear functionals $X\in\HH^*$ such that $\langle X,\sigma\cdot\tau\rangle=\langle X,\sigma\rangle\langle X,\tau\rangle$.
For each $n\in\N$ the finite-dimensional vector space $\HH_{n}$ spanned by the set $\cal F_n$ of forests with at most $n$ nodes is a subcoalgebra of $\HH $, hence its dual is an algebra under the convolution product, and we let $G_{n}$ be the set of characters on $\HH_{n}$.

We have already defined branched rough paths in \Cref{def:brp}.
\Cref{prp:genrphold} yields the following characterization
\begin{prp}
  A path $X\colon[0,1]^2\to G^N$ is a branched rough path if and only if $\bb X_t\coloneq X_{0t}$ is $\gamma$-H\"older path with respect to the
  distance $\rho_N$ defined in \eqref{eq:rho}.
\end{prp}

Directly applying \Cref{thm:charext} to the Butcher-Connes-Kreimer Hopf algebra $\HH$ we obtain
\begin{crl}
  Given $\gamma\in\,]0,1[$ with $\gamma^{-1}\notin\N$ and a family of $\gamma$-Hölder paths $(x^i:i=1,\dotsc,d)$, there exists a branched rough path $X$ above $(x^i:i=1,\dotsc,d)$, i.e. $X\colon[0,1]^2\to G^N$ is such that
  $\langle X_{st},\dtR<i>\rangle = x_t^i-x_s^i$ for all $i=1,\dotsc,d$.
  \label{crl:branchext}
\end{crl}

\begin{rmk}
  Given the level of generality in which \Cref{thm:charext} is developed, our results also apply to the case when $\HH$ is a combinatorial Hopf algebra as defined in \cite{Curry2018}.
  In particular, we also have a construction theorem for \emph{planarly branched rough paths} \cite{Curry2018} which are characters over Munthe-Kaas and Wright's Hopf algebra of Lie group integrators \cite{Munthe-Kaas2008}.
\end{rmk}

\subsection{Anisotropic geometric rough paths}
\label{sse:aRP}
We now apply our results to another class of rough paths which we call \emph{anisotropic geometric rough paths} (aGRPs for short).
L. Gyurkó introduced a similar concept in \cite{Gyurko2016}, which he called $\Pi$-rough paths; unlike us, he uses a ``primal'' presentation, i.e. paths taking values in the tensor algebra $T(\R^d)$, and $p$-variation norms rather than H\"older norms. Geometric rough paths over a inhomogeneous (or anisotropic) set of paths can be traced back to Lyons' original paper \cite{Lyons1998}.

As in the geometric case, see \Cref{sss:grp}, fix a finite alphabet $A=\{1,\dotsc,d\}$ and denote by $\mathrm M(A)$ the free monoid generated by $A$. We denote again by $\HH\coloneq T(A)$ the shuffle Hopf algebra over the alphabet $A$.

Let $(\gamma_a:a\in A)$ be a sequence of real numbers such that $0<\gamma_a<1$ for all $a$, and let $\hat{\gamma}=\min_{a\in A}\gamma_a$.
For a word $v=a_1\dotsm a_k\in\mathrm M(A)$ of length $k$ define
\[
  \omega(v)\coloneq{\gamma_{a_1}+\dotsc+\gamma_{a_k}}
\]
and observe that $\omega$ is additive in the sense that $\omega(uv)=\omega(u)+\omega(v)$ for each pair of words $u,v\in\mathrm M(A)$.
The set
\[
\f L\coloneq \{v\in\mathrm M(A):\omega(v)\le 1\}
\]
is finite; if $\hat N\coloneq\lfloor\hat{\gamma}^{-1}\rfloor$ then $\f L\subset \HH_{\hat N}$.
In analogy with \Cref{lmm:subc}, the additivity of $\omega$ implies
\begin{lmm}
  The subspace $\HH_{\mathrm a}\subset \HH_{\hat N}$ spanned by $\f L$ is a subcoalgebra of $(\HH,\bar{\Delta},\varepsilon)$.
\end{lmm}

Consequently, we will consider the dual algebra $(\HH_{\mathrm a}^*,\star,\varepsilon)$. In this case, we \emph{define} $\g_{\mathrm a}$ to be the space of truncated infinitesimal characters on $\HH_{\mathrm a}$, namely
the linear functionals $\alpha\in \HH_{\mathrm a}^*$ such that
\[
\langle\alpha, x\shuffle y\rangle=\langle\alpha,x\rangle\langle\varepsilon,y\rangle+\langle\varepsilon,x\rangle\langle\alpha,y\rangle
\]
for all $x,y\in\HH_{\mathrm a}$ such that $x\shuffle y\in\HH_{\mathrm a}$,
 and let $G_{\mathrm a}\coloneq\{X=\left.\exp_{\hat N}(\alpha)\right|_{\HH_{\mathrm a}}: \, \alpha\in\g_{\mathrm a}\}$.
As before, there is a canonical injection $\HH_{\mathrm a}^*\hookrightarrow \HH^*$ so we suppose that $\langle X,v\rangle=0$ for all $X\in \HH^*$ and $v\not\in\f L$.

For each $\lambda>0$ there is a unique coalgebra automorphism $\Omega_\lambda\colon\HH\to \HH$ such that $\Omega_\lambda a=\lambda^{\gamma_a/\hat{\gamma}}a$ for all $a\in A$. We also define $\|\cdot\|\colon G_{\mathrm a}\to\R$,
\begin{equation}
  \|X\|\coloneq\max_{v\in\f L} \, |\langle X,v\rangle|^{\hat{\gamma}/\omega(v)}.
 \label{eq:anorm}
\end{equation}
As at the end of \Cref{sec:2}, $(\Omega_\lambda)_{\lambda>0}$ is a one-parameter family of Lie-algebra automorphisms of $\g_{\mathrm a}$ and
$\|\Omega_\lambda X\|=\lambda\|X\|$ for all $\lambda>0$ and $X\in G_{\mathrm a}$, namely $\|\cdot\|$ is a homogeneous norm on $G_{\mathrm a}$. However, unlike $\lrn{\cdot}$ this norm is not subadditive and it therefore does not define a distance on $G_{\mathrm a}$.

\subsubsection{Signatures}
In order to construct an appropriate metric on $G_{\mathrm a}$ we consider \emph{signatures} of smooth paths. We observe that $A\subset\f L$.
Let $x=(x^a:a\in A)$ be a collection of (piecewise) smooth paths, and define a map $S(x)\colon[0,1]^2\to \HH^*$ by
\[ \langle S(x)_{st},v\rangle\coloneq \int_s^t \mathrm{d}x^{v_k}_{s_k}\int_{s}^{s_k}\mathrm{d}x^{v_{k-1}}_{s_{k-1}}\dotsm\int_s^{s_2}\mathrm{d}x^{v_1}_{s_1}. \]
In his seminal work \cite{Chen1977}, K. T. Chen showed that $S(x)$ is a character of $(T(A),\shuffle)$; in particular, $\left.S(x)_{st}\right|_{\HH_{\mathrm a}}\in G_{\mathrm a}$.

Consider the metric $d_{\mathrm a}(X,Y)=\sum_{a\in A}|\langle X-Y,a\rangle|^{\hat{\gamma}/\gamma_a}$ on $\HH_{(1)}^*$, where we recall that $\HH_{(1)}$ is the vector space spanned by $A$.
The \emph{anisotropic length} of a smooth curve $\theta\colon[0,1]\to \HH_1^*$ is defined to be its length with respect to this metric and will be denoted by $L_{\mathrm a}(\theta)$.
Observe that since $d_{\mathrm a}(\Omega_\lambda X,\Omega_\lambda Y)=\lambda d_{\mathrm a}(X,Y)$ we have that $L_{\mathrm a}(\Omega_\lambda\theta)=\lambda L_{\mathrm a}(\theta)$.

We now define a homogeneous norm (see the end of \Cref{sec:2}) $|\cdot|_{\rm CC}\colon G_{\mathrm a}\to\R_+$, called the \emph{anisotropic Carnot--Carath\'eodory norm}, by setting \[ |X|_{\rm CC}\coloneq \inf\{L_{\mathrm a}(x): x^a\in C^\infty, S(x)_{01}=X\}. \]
Since curve length is invariant under reparametrization in any metric space we obtain, as in \cite[Section 7.5.4]{Friz2010}:
\begin{prp}
  The infimum defining the anisotropic Carnot--Carath\'eodory norm is finite and attained at some minimizing path $\hat{x}$.
\end{prp}

\begin{prp}
  The anisotropic Carnot--Carath\'eodory norm is homogeneous, that is, $|\Omega_\lambda X|_{\rm CC}=\lambda|X|_{\rm CC}$.
\end{prp}
\begin{proof}
  Let $\hat{x}$ be the curve such that $|X|_{\rm CC}=L_a(\hat{x})$. For any $\lambda>0$ and word $v\in\f L$ we have
  \begin{align*}
    \langle S(\Omega_\lambda\hat{x})_{01},v\rangle &= \lambda^{\omega(v)/\hat\gamma}\langle S(\hat{x})_{01},v\rangle
    = \langle \Omega_\lambda S(\hat{x})_{01}, v\rangle
    = \langle \Omega_\lambda X,v\rangle,
  \end{align*}
  thus $|\Omega_\lambda X|_{\rm CC}\le L_{\mathrm a}(\Omega_\lambda\hat{x})=\lambda L_{\mathrm a}(\hat{x})=\lambda|X|_{\rm CC}$.
  The reverse inequality is obtained by noting that $X=(\Omega_{\lambda^{-1}}\circ\Omega_\lambda) X$.
\end{proof}

The anisotropic Carnot--Carath\'eodory norm can also be seen to satisfy $|X|_{\rm CC}=|X^{-1}|_{\rm CC}$ and
$|X\star Y|_{\rm CC}\leq |X|_{\rm CC}+|Y|_{\rm CC}$ for all $X,Y\in G_{\mathrm a}$, see e.g. the proof of \cite[Proposition 7.40]{Friz2010};
hence it induces a left-invariant metric $\rho_{\mathrm a}(X,Y)\coloneq|X^{-1}\star Y|_{\rm CC}$ on $G_{\mathrm a}$. 
Moreover, arguing
as in the proof of \cite[Theorem 7.44]{Friz2010} we see that there exist positive constants $c,C$ such that 
\begin{equation}\label{eq:eq} 
c|X|_{\rm CC}\leq \| X\|\leq C|X|_{\rm CC}, \qquad \forall \ X\in G_{\mathrm a}.
\end{equation}

\begin{dfn}\label{def:ageo}
  An \emph{anisotropic geometric $\gamma$-rough path}, with $\gamma=(\gamma_a, \, a\in A)$, is a map $X\colon[0,1]^2\to G_{\mathrm a}$ which satisfies
  \begin{enumerate}
  \item the Chen rule $X_{su}\star X_{ut}=X_{st}$ for all $(s,u,t)\in[0,1]^3$,
   \item the bound $|\langle X_{st},v\rangle|\lesssim|t-s|^{\omega(v)}$ for all $v\in\f L$.
   \end{enumerate}
\end{dfn}

\begin{prp}
  Anisotropic geometric $\gamma$-rough paths are in one-to-one correspondence with $\hat{\gamma}$-Hölder paths $\bb{X}\colon[0,1]\to (G_{\mathrm a}, \rho_{\mathrm a})$ with $\bb X_0=\1$.
\end{prp}
\begin{proof}
  Let $X$ be an anisotropic geometric $\gamma$-rough path and $v$ a word. By definition we have that $|\langle X_{st},v\rangle|\lesssim|t-s|^{\omega(v)}$, hence $\|X_{st}\|\lesssim|t-s|^{\hat{\gamma}}$.
  The equivalence between $\|\cdot\|$ and $|\cdot|_{\rm CC}$ of \eqref{eq:eq}  implies that $\rho_{\mathrm a}(\bb{X}_s,\bb{X}_t)=|X_{st}|_{\rm CC}\lesssim|t-s|^{\hat{\gamma}}$, hence $t\mapsto \bb{X}_t$ is $\hat\gamma$-Hölder with respect to $\rho_{\mathrm a}$.
  The other direction follows in a similar manner.
\end{proof}

\Cref{thm:charext} also applies to this situation, and we obtain the following
\begin{crl}
  Let $(\gamma_a:a\in A)$ be real numbers in $]0,1[$ such that $1\not\in\sum_{a\in A}\gamma_a\N$.
  Let $(x^a:a\in A)$ be a collection of real-valued paths such that $x^a$ is $\gamma_a$-Hölder. Then
 there exists an anisotropic geometric $\gamma$-rough path $X$ such that $\langle X_{st},a\rangle=x^a_t-x^a_s$ for all $a\in A$.
  \label{crl:agrpext}
\end{crl}
\begin{proof}We start by constructing the \emph{homogeneous} geometric rough path $X$ given by the $\hat\gamma$-H\"older path $\bb X\colon[0,1]\to G^{\hat N}$ of
\Cref{crl:3.8}. Then we restrict $\bb X$ to $\HH_{\mathrm a}\subset\HH_{\hat N}$ and we show that on this space it satisfies the stronger bound $|\langle X_{st},v\rangle|\lesssim|t-s|^{\omega(v)}$ for all $v\in\f L$.

Recalling the proof of Theorem \Cref{thm:charext}, we consider $v\in\HH_{n}\cap \HH_{\mathrm a}$, and we proceed by recurrence on $n$. 
For $n=0$ there is nothing to prove. Suppose we have proved the result for $n$ and let  $v\in\HH_{n+1}\cap \HH_{\mathrm a}$. In this case
\[
\begin{split}
\langle X^{n+1}_{st},v\rangle &= \langle \exp_{n+1}(L_{st}+Z_{st}),v\rangle = 
\sum_{i=0}^{n+1}\frac1{i!} \langle (L_{st})^{i\star}\star(\varepsilon+Z_{st}),v\rangle
\\ &=\sum_{i=0}^{n+1}\frac1{i!} \langle (L_{st})^{i\star},v\rangle + \langle Z_{st},v\rangle
= \langle X^{n}_{st},v\rangle + \frac1{(n+1)!} \langle (L_{st})^{(n+1)\star},v\rangle + \langle Z_{st},v\rangle.
\end{split}
\]
We want to prove now that 
\begin{equation}\label{eq:esti}
\left|\langle X^{n+1}_{t^m_kt^m_{k+1}},v\rangle\right| \lesssim 2^{-m\omega(v)}, \qquad \forall \, m\geq 0,
\ k=0,\ldots,2^m-1.
\end{equation}
 For $m\geq 0$ set
 \[ b_m\coloneq 2^{m\omega(v)}\max_{k=0,\dotsc,2^m-1}\left|\langle Z^m_{t^m_kt^m_{k+1}},v\rangle\right|. \]
  Then, for $s=t^m_k$, $u=t^m_{k+1}$ and $t=t^m_{k+2}$ and $v=v_1\cdots v_{n+1}$
  \begin{equation*}
    |\langle\BCH_{(n+1)}(L_{su},L_{ut}),v\rangle| \le \sum_{i+j=n+1}\frac{1}{i!j!}\!\sum_{\sigma\in S_{n+1}}\!\!|a_\sigma|\prod_{p=1}^i|\langle L_{su},v_{\sigma^{-1}(p)}\rangle|\prod_{q=i+1}^{n+1}\!\!|\langle L_{ut},v_{\sigma^{-1}(q)}\rangle|.
  \end{equation*}
Now, since $v_j\in\HH_{(1)}$ for all $j=1,\dotsc,n+1$ we actually have that by the assumption $x^a\in C^{\gamma_a}$
  \[ |\langle L_{su},a\rangle|= \left|x^a_u-x^a_{s}\right|\lesssim 2^{-m\gamma_a} \]
  and we have a similar estimate for $L_{ut}$ instead of $L_{su}$.
  Therefore we obtain that
  \[ \left|\langle\BCH_{(n+1)}(L_{su},L_{ut}),v\rangle\right| \lesssim 2^{-m\omega(v)}. \]
  Therefore, from \eqref{eq:Zdef0} we get
  \begin{equation*}
    b_{m}\le 2^{m(\omega(v)-1)}b_{m-1}+C, \qquad m\geq 1,
  \end{equation*}
  hence since $b_0=0$ we can show by recurrence on $m\geq 0$
  \[ b_m\le C\sum_{j=0}^{m-1}2^{-j(1-\omega(v))}. \]
  Since we are in the regime where $\omega(v)<1$ (here we use that $1\not\in\sum_{a\in A}\gamma_a\N$) we obtain that 
  \[
   \sup_{m\ge 0}b_m \le\frac{C}{1-2^{\omega(v)-1}}.
  \]
  Therefore
  \[
  \left|\langle Z^m_{t^m_kt^m_{k+1}},v\rangle\right| \lesssim 2^{-m\omega(v)}, \qquad {m\geq 0}, \ {k=0,\dotsc,2^m-1}.
  \]
Analogously, since $L_{st}\in\g$, arguing as in \eqref{eq:starsta} we have
\[
\langle (L_{st})^{(n+1)\star},v\rangle = \prod_{i=1}^{n+1} \langle L_{st},v_i\rangle
= \prod_{i=1}^{n+1} \left( x^{v_i}_{t}-x^{v_i}_{s}\right)
\Longrightarrow\left|\langle (L_{st})^{(n+1)\star},v\rangle\right| \lesssim 2^{-m\omega(v)}
\]
and \eqref{eq:esti} is proved. This implies that $\|X^{n+1}_{t^m_kt^m_{k+1}}\|\lesssim 2^{-m\hat\gamma}$
and by equivalence of homogeneous norms \eqref{eq:eq} we obtain 
 \[
  \rho_{\mathrm a}\left(\bb X^{n+1}_{t^m_k},\bb X^{n+1}_{t^m_{k+1}}\right)\lesssim 2^{-m\hat\gamma}.
  \]
Then we can use \Cref{lmm:LV} and obtain that the path $\bb X^{n+1}$ constructed
in the proof of \Cref{thm:charext} is in fact $\hat\gamma$-H\"older path with
values in $G_{\mathrm a}$.
\end{proof}

\section{The Hairer-Kelly construction}
\label{sse:brp+}
In this section we develop further results specifically for branched rough paths as introduced in \Cref{sec:brp} by using our general results from \Cref{sse:RPthm}.
We analyze in detail the Hairer-Kelly map introduced in \cite{Hairer2014}, which plays a very important role in our construction, and we use it to prove \Cref{thm:intro.trans} and \Cref{crl:param}.

\subsection{The Hairer--Kelly map}
\label{sss:RPHK}
Recall that $\T$ denotes the set of all decorated rooted trees, $\F$ denotes the collection of all decorated rooted forests, and $\HH_{\mathrm{BCK}}$ is the Butcher--Connes--Kreimer Hopf algebra. As in \Cref{sec:brp}, $\Delta$ denotes the Connes--Kreimer coproduct on $\HH_{\mathrm{BCK}}.$
For each $n\in\N$, $n\geq 1$, we denote by $\T_n$ the set of (non-empty) trees with at most $n$ vertices. 

Recall also from \Cref{sss:grp} that given an alphabet $A$ we denote by $T(A)$ the shuffle Hopf algebra generated by $A$, and that $\overline \Delta$ denotes the deconcatenation coproduct on it. We fix $N\in\N$ and we consider the shuffle Hopf algebras $T(\T)$ and $T(\T_N)$, 
namely we choose as letters of our alphabet the (non-empty) decorated rooted trees (respectively rooted trees with with at most $N$ vertices). 
Note that we can identify every non-empty tree $\tau\in\T$ with the
word in $T(\T)$ composed by the single letter $\tau$.
We also remark that, in order to avoid confusion with the forest product on $\HH_{\mathrm{BCK}}$ we denote the concatenation of letters in $T(\T)$ by a tensor symbol. 

We note that $T(\T)$ and $T(\T_N)$ admit two different natural gradings, both of which make them locally finite graded Hopf-algebras. One grading,
as in \Cref{sss:grp}, is given by the number of letters (trees) of each word, namely the degree of $v=\tau_1\otimes\dotsm\otimes\tau_k$ is $k$. The
other grading is given by the sum of the number of nodes of each letter (tree), namely the degree of $v=\tau_1\otimes\dotsm\otimes\tau_k$ is 
$|\tau_1|+\dotsb+|\tau_k|$, where we recall that forests and trees are graded in $\HH_{\mathrm{BCK}}$ by the number of nodes, with the notation $|\tau|=\#N_\tau$. We remark the latter grading is always greater or equal to the former.  As an example, take $v= \dtR<i>\otimes\dtI<jk>$; then, as a word $v$ has length 2 but the total number of nodes is 3.

We recall the following result from \cite[Lemma 4.9]{Hairer2014}.
\begin{lmm}
We grade $T(\T)$ according to the number of nodes. Then there exists a graded morphism of Hopf algebras $\psi\colon\HH_{\mathrm{BCK}}\to T(\T)$ satisfying $\psi(\tau)=\tau+\psi_{n-1}(\tau)$ for all $\tau\in\T_n$, where $\psi_{n-1}$ denotes the projection of $\psi$ onto $T(\T_{n-1})$. 
  \label{lmm:HK}
\end{lmm}
We call $\psi$ the Hairer-Kelly map. Since $\psi$ is graded, for any forest $\tau\in\F$ the image $\psi(\tau)$ is a sum of words of the form
$\tau_1\otimes\dotsm\otimes\tau_k$ where all terms satisfy $|\tau_1|+\dotsb+|\tau_k|=|\tau|$.
Observe that since $\psi$ is a Hopf algebra morphism, in particular a coalgebra morphism, then
\[ (\psi\otimes\psi)\Delta'\tau=\bar{\Delta}'\psi(\tau)=\bar{\Delta}'\psi_{n-1}(\tau), \qquad \tau\in\T_n, \]
since trees are primitive elements in $T(\T)$, being single-letter words.
From the proof of \cite[Lemma 4.9]{Hairer2014} we are able to see that in fact $\psi_{n-1}$ is given by the recursion $\psi_{n-1} = m_{\otimes}(\psi\otimes{\id})\Delta'$ on the linear span of $\T_n$, see also \cite[Definition 1, section 6]{BCF}.

\begin{xmp}
  Here are some examples of the action of $\psi$ on some trees:
  \begin{align*}
    \psi\left( \dtR<i> \right) &= \dtR<i>, \qquad 
    \psi\left( \dtR<a>\dtR<b> \right) = \psi\left( \dtR<a> \right)\shuffle\psi\left( \dtR<b> \right)=\dtR<a>\otimes\dtR<b>+\dtR<b>\otimes\dtR<a>, \qquad
    \psi\left( \dtI<ab> \right) = \dtI<ab> + \dtR<b>\otimes\dtR<a>\\
    \psi\left( \Forest{decorated[a [c [d] ] [b] ]} \right) &= \Forest{decorated[a [c [d] ] [b] ]} + \dtR<b>\otimes\dtII<acd> + \dtR<d>\otimes\dtV<acd> + \dtI<cd>\otimes\dtI<ab> + \dtR<d>\otimes\dtR<c>\otimes\dtI<ab> + \dtR<d>\otimes\dtR<b>\otimes\dtI<ac> + \dtR<b>\otimes\dtR<d>\otimes\dtI<ac>\\
    &\quad + \dtI<cd>\otimes\dtR<b>\otimes\dtR<a> + \dtR<b>\otimes\dtI<cd>\otimes\dtR<a> + \dtR<d>\otimes\dtR<c>\otimes\dtR<b>\otimes\dtR<a> + \dtR<d>\otimes\dtR<b>\otimes\dtR<c>\otimes\dtR<d>
    \\ & \quad + \dtR<b>\otimes\dtR<d>\otimes\dtR<c>\otimes\dtR<a>.
  \end{align*}
  \label{xmp:psi}
\end{xmp}

\subsection{A special class of anisotropic geometric rough paths}
We have already discussed anisotropic geometric rough paths (aGRPs) in \Cref{sse:aRP}. For the Hairer-Kelly construction
we need a very particular subclass of aGRPs, where the base paths $(x^a)_{a\in A}$ are such that each $x^a$ is
${\gamma_a}$-H\"older and there exists $\gamma\in\,]0,1[$ and $(k_a)_{a\in A}\subset\N$ such that $\gamma_a=k_a\gamma$; therefore the H\"older exponents are all integer multiples of a fixed exponent $\gamma$.

We may of course apply the extension result of \Cref{crl:agrpext}, but it turns out that in this setting we
can avoid using the Carnot-Carath\'eodory distance and rather use a more explicit metric, which is a simple
generalization of the homogeneous case \eqref{eq:rho}.

We have already seen that the space $\HH\coloneq T(\T_N)$ can be graded in two ways. We can even define a {\it bigrading} on this
space: for $1\leq n\leq N$ and 
$n\leq j\leq nN$, we define the space $\HH_{(n,j)}$ as the linear span of the words $\tau_1\otimes\cdots\otimes\tau_n\in T(\T_N)$ such that 
$|\tau_1|+\dotsb+|\tau_n|=j$. Then, in analogy with \eqref{eq:coprod}, we have
\[
\shuffle\colon\HH_{(n,j)}\otimes\HH_{(m,h)}\to\HH_{(n+m,j+h)}, \qquad \bar\Delta\colon\HH_{(n,j)} \to \bigoplus_{p=0,\ldots,n, \ q = 1,\ldots,j-1}\HH_{(p,q)}\otimes\HH_{(n-p,j-q)}.
\]
Then, recalling that $ \HH_0=\R\1$, we set
\[
\HH_{N,N} \coloneq\HH_0\oplus \bigoplus_{n=1}^{N}\bigoplus_{j=n}^{N}\HH_{(n,j)}.
\]
In other words, $\HH_{N,N}$ is the linear span of all words $\tau_1\otimes\cdots\otimes\tau_n$ with $n\leq N$ and $|\tau_1|+\dotsb+|\tau_n|\leq N$.
Therefore, analogously to \eqref{eq:hhx} and \eqref{eq:hngrad}, we have decompositions
\[
\HH_{N,N}\ni x=x_0+\sum_{n=1}^N\sum_{j=n}^{N} x_{n,j}, \qquad \HH_{N,N}^*\ni\alpha=\alpha_{(0)}+\sum_{n=1}^N\sum_{j=n}^{N} \alpha_{(n,j)}, \qquad \alpha_{(n,j)}(x)=\alpha(x_{n,j}).
\]
We define now $\g^{N,N}$ as the space of truncated characters on $\HH_{N,N}$, namely of all linear $\alpha\colon\HH_{N,N}\to\R$ such that
\[
\langle\alpha,x\shuffle y\rangle=\langle\alpha,x\rangle\langle\varepsilon,y\rangle+\langle\varepsilon,x\rangle\langle\alpha,y\rangle
\]
for all $x,y\in\HH_{N,N}$ such that $x\shuffle y\in\HH_{N,N}$. Moreover we define $G^{N,N}\coloneq\exp_N(\g^{N,N})\subset\HH^*_N$.
Then we set in analogy with \eqref{lmm:convcomp} for $X\in G^{N,N}$
\begin{equation}
  |X| \coloneq N\max_{n=1,\dotsc,N}\left(\max_{j=n,\dotsc,N}\left( j!\lrn{X_{(n,j)}} \right)^{1/j}+\max_{j=n,\dotsc,N}\left( j!\lrn{\left(X^{-1}\right)_{(n,j)}} \right)^{1/j}\right),
  \label{eq:homnorm2}
\end{equation}
and we can see that
\begin{lmm}
The map $G^{N,N}\times G^{N,N}\ni(X,Y)\mapsto \rho^{N,N}(X,Y)\coloneq|X^{-1}\star Y|\in\R$ defines a distance on $G^{N,N}$.
\end{lmm}
\begin{proof}
We only have to check the triangular inequality, which is equivalent to the sub-additivity property $|X\star Y|\leq |X|+|Y|$ for
all $X,Y\in G^{N,N}$. Arguing as in the proof of \Cref{prp:subadd}
\[  \begin{split}
    \lrn{(X\star Y)_{(n,j)}} &\leq \sum_{m=0}^n\sum_{i=1}^{j-1}\lrn{X_{(m,j)}}\lrn{Y_{(n-m,j-i)}}
    \\ & \le N  \frac{1}{j!} \frac1{N^2} \sum_{i=0}^j\binom{j}{i}|X|^i|Y|^{j-i}
    = \frac1{N} \frac{1}{j!}(|X|+|Y|)^{j}
  \end{split} \]
  whence the result.
\end{proof}
Let $\gamma\in\,]0,1[$ and $N\coloneq\lfloor\gamma^{-1}\rfloor$.
In accordance with \Cref{def:ageo},  an \emph{anisotropic geometric $\gamma$-rough path} in this setting is a map $X\colon[0,1]^2\to G^{N,N}$ which satisfies
  \begin{enumerate}
  \item the Chen rule $X_{su}\star X_{ut}=X_{st}$ for all $(s,u,t)\in[0,1]^3$,
   \item $|\langle X_{st},v\rangle|\lesssim|t-s|^{j\gamma}$ for all $v\in\HH_{(n,j)}$ with $1\leq n\leq N$ and $j\leq N$.
   \end{enumerate}
Then, arguing as in \Cref{prp:genrphold}, it is easy to show that $\bb X\colon[0,1]\to G^{N,N}$ is $\gamma$-Hölder with respect to the metric $\rho^{N,N}$
if and only if $X\colon[0,1]^2\to G^{N,N}$, defined as $X_{st}\coloneq\bb X_s^{-1}\star\bb X_t$, 
is an {anisotropic geometric $\gamma$-rough path} with $\gamma_v=j\gamma$ for $v=\tau_1\otimes\cdots\otimes\tau_n$ with $n\leq N$ and $|\tau_1|+\dotsb+|\tau_n|=j\leq N$.

The next result is the analog of \Cref{crl:agrpext} in this setting. The proof is the same, with one exception: we can use the explicit norm 
\eqref{eq:homnorm2} rather than the Carnot-Carath\'eodory norm $|\cdot|_{\rm CC}$ and we do not need the equivalence of norms result \eqref{eq:eq}.
\begin{prp}\label{thm:main2}
Given $\gamma\in\,]0,1[$ with $\gamma^{-1}\notin\N$ and a collection of paths $x^\tau\colon[0,1]\to\R$, $\tau\in\T_N$, such that $x^\tau\in C^{\gamma|\tau|}$, there exists a $\gamma$-H\"older path $\bb X\colon[0,1]\to G^{N,N}$ such that $\langle \bb X,\tau\rangle =x^\tau$ for all $\tau\in\T_N$.
\end{prp}

\begin{crl}
In the setting of \Cref{thm:main2}, let $(g^\tau:\tau\in \T_N)$ be
a collection of functions with $g^\tau\in C^{\gamma|\tau|}$. Set $\bar{x}^\tau_t = x^\tau_t + g^\tau_t$ and denote by $gX$ the anisotropic geometric $\gamma$-rough path constructed in \Cref{thm:main2} above the path
  \[ \bar{x}_t = \sum_{\tau\in \T_N}\bar{x}_t^\tau \,\tau\in\HH_{(1)}, \qquad t\in[0,1]. \]
  Then, for any two such functions $g$ and $g'$ we have that $g'(gX)=(g+g')X$.
\end{crl}
\begin{proof}
Let $g,g'$ be two collections of functions as in the statement of the theorem. We have the identity \[ \langle [g'(gX)]_t,\tau\rangle = \langle (gX)_t,\tau\rangle + (g')^\tau_t = x^\tau_t + g^\tau_t + (g')^\tau_t = \langle[(g'+g)X]_t,\tau\rangle. \]
Since both $g'(gX)$ and $(g'+g)X$ are constructed iteratively by adding at each step a function $Z$ satisfying \eqref{eq:Zdef} on the dyadics, if we let $L^n$ and $\bar{L}^n$ denote the logarithms corresponding to $g'(gX)$ and $(g'+g)X$, \Cref{lmm:phik} and the previous identity imply that \[ \BCH_{n+1}(L^n_{su},L^n_{ut}) =\BCH_{n+1}(\bar{L}^n_{su},\bar{L}^n_{ut}) \]
and so $g'(gX)=(g'+g)X$.
\end{proof}

\subsection{Branched rough paths are anisotropic geometric rough paths}
The next theorem is almost the same statement as Theorem 4.10 in \cite{Hairer2014}, the only difference being
that we construct an \emph{anisotropic} geometric rough path $\bar X$ while Hairer-Kelly need only that $\bar X$ is geometric in the usual sense (see also \cite[Remark 4.14]{Hairer2014}.
\begin{thm}
Let $\gamma\in\,]0,1[$ with $\gamma^{-1}\notin\N$, and
  let $X$ be a branched $\gamma$-rough path. There exists an anisotropic geometric rough path $\bar{X}\colon[0,1]^2\to G^{N,N}$ with exponents $\gamma=(\gamma_\tau=\gamma|\tau|, \tau\in\T_N)$, and such that
  \[
  \langle X,\tau\rangle = \langle\bar{X},\psi(\tau)\rangle, \qquad \forall\, \tau\in\F_N.
  \]
  \label{thm:b2ag}
\end{thm}
\begin{proof}
  We construct $\bar{X}$ iteratively as follows.
  Let $\bar{X}^{(1)}$ be the anisotropic geometric rough path indexed by $\T_1=\{\dtR<1>,\dotsc,\dtR<d>\}$ over the paths $(x^i_t\coloneq \langle X_t,\dtR<i>\rangle:i=1,\dotsc,d)$ with exponents $(\gamma_{\dtR<i>}=\gamma)$ given by \Cref{thm:main2} (alternatively we could use have used \Cref{thm:charext} since all the exponents are equal). This will give us an anisotropic rough path path $X\colon[0,1]^2\to G_{\mathrm a}(\T_1)$ with exponents $(\gamma_\tau=\gamma, \tau\in\T_1)$.

  Suppose we have constructed anisotropic geometric rough paths $\bar{X}^{(k)}\colon[0,1]^2\to G_{\mathrm a}(\T_k)$ over the paths $(x^\tau:\tau\in\T_k)$ such that $x^\tau_t-x^\tau_s=\langle X_{st},\tau\rangle-\langle\bar{X}_{st}^{(k-1)},\psi_{k-1}(\tau)\rangle$ for $k=1,\dotsc,n$.
  This is true for $n=1$ by the previous paragraph, since $\psi(\dtR<i>)=\dtR<i>$ for all $i=1,\dotsc,d$.

  If we let $F_{st}^\tau=\langle X_{st},\tau\rangle$ and $G_{st}^\tau=\langle \bar{X}^{(n)}_{st},\psi_n(\tau)\rangle$ for $\tau\in\T_{n+1}$ we have, by Chen's rule, that
  \begin{align*}
    \delta F^\tau_{sut} &= \langle X_{su}\otimes X_{ut},\Delta'\tau\rangle
    = \langle \bar{X}^{(n)}_{su}\circ\psi\otimes\bar{X}^{(n)}_{ut}\circ\psi,\Delta'\tau\rangle.
  \end{align*}
  Since $\psi$ is in particular a coalgebra morphism between $(\cal H,\Delta)$ and $(T(\T_N),\bar{\Delta})$ we obtain the identity $\delta F^\tau_{sut}=\langle\bar{X}^{(n)}_{su}\otimes\bar{X}^{(n)}_{ut},\bar{\Delta}'\psi(\tau)\rangle$, which then, by \Cref{lmm:HK} becomes
  \begin{equation}
\label{eq:deltaF}
    \delta F^\tau_{sut} = \langle\bar{X}^{(n)}_{su}\otimes\bar{X}^{(n)}_{ut},\bar{\Delta}'\psi_n(\tau)\rangle = \delta G^\tau_{sut}.
  \end{equation}
  since every $\tau\in\T$ is primitive in $(T(\T_N),\bar\Delta)$ being a single-letter word.

  The finite increment operator $\delta$ has the following property: if $J\colon[0,1]^2\to\R$ is such that $\delta J=0$ then there exists $f\colon[0,1]\to\R$ such that $J_{st}=f_t-f_s$, and the function $f$ is unique up to an additive constant shift, see also \cite[formula (5)]{Gubinelli2010}.
  Thus, by this fundamental property, for each $\tau\in\T_{n+1}$ there exists a function $x^\tau\colon[0,1]\to\R$ such that $x_t^\tau-x_s^\tau=F^\tau_{st}-G^\tau_{st}$ and then
  \begin{align*}
    |x_t^\tau-x_s^\tau| &\leq |\langle X_{st},\tau\rangle|+|\langle\bar{X}^{(n)}_{st},\psi_n(\tau)\rangle|
    \lesssim|t-s|^{\gamma|\tau|}
  \end{align*}
  since $\psi_n(\tau)$ preserves the number of nodes by \Cref{lmm:HK}.

  Repeatedly using \Cref{thm:main2} we obtain an anisotropic geometric rough path $\bar X^{(n+1)}\colon[0,1]\to G_{\mathrm a}(\T_{n+1})$ over $(x^\tau:\tau\in\T_{n+1})$ whose restriction to $T(\T_n)$ coincides with $\bar X^{(n)}$.

  Finally notice that if $\tau\in\T_{n+1}$ is a tree then
  \begin{align*}
    \langle\bar{X}^{(n+1)}_{st},\psi(\tau)\rangle &= \langle\bar{X}^{(|\tau|)}_{st},\tau\rangle+\langle\bar{X}^{(|\tau|)}_{st},\psi_{|\tau|-1}(\tau)\rangle\\
    &= x^\tau_t-x^\tau_s+\langle X_{st},\tau\rangle-(x^\tau_t-x^\tau_s)
    = \langle X_{st},\tau\rangle
  \end{align*}
  and the corresponding identity for arbitrary forests follows by multiplicativity.
  The anisotropic geometric rough path sought for is $\bar{X}=\bar{X}^{(N)}$.
\end{proof}
We note that our proof is shorter and simpler than that of \cite[Theorem 4.10]{Hairer2014}, so we will now dedicate a few paragraphs to highlight the differences between our approach and that of Hairer and Kelly.
They define first
\[
\hat{\bb X}_t^1=\exp_N\left(\sum_{a\in A}x^a_t\,\dtR<a> \right)\in G^N(\T_1)
\]
then they note that this is not $\gamma$-H\"older with values in $G^N(\T_1)$, but it is $\gamma$-H\"older with
values in $G^N(\T_1)/K_1$, where $K_1\coloneq\exp_N(W_2+\cdots+W_N)$, see \eqref{WW}. By the Lyons-Victoir extension theorem there exists a $\gamma$-H\"older path $\bar{\bb X}_t^1\to G^N(\T_1)$
such that $\pi_{G^N(\T_1)\to G^N(\T_1)/K_1}(\bar{\bb X}^1)=\hat{\bb X}^1$. Then, in order to add a new tree $\tau$ with $|\tau|=2$, they define
\[
(\delta \bar X^\tau)_{st}=\langle X_{st},\tau\rangle-\langle \bar{X}^{(1)}_{st},\psi_1(\tau)\rangle
\]
and this defines the new function $t\mapsto \langle \hat{\bb X}_{t},\tau\rangle$. Then they
define
\[
\hat{\bb X}_t^2=\exp_N\left(\sum_{a\in A}x^a_t\,\dtR<a>+\sum_{|\tau|=2} \langle \hat{\bb X}_{t},\tau\rangle\, \tau \right)
\in G^N(\T_2)
\]
and again they note that this path is not $\gamma$-H\"older with values in $G^N(\T_2)$, but it is with values in $G^N(\T_2)/K_2$, where $K_2\coloneq\exp_N(W_3+\cdots+W_N)$, and again the Lyons-Victoir extension theorem yields
a $\gamma$-H\"older path $\bar{\bb X}_t^2\to G^N(\T_2)$
such that $\pi_{G^N(\T_2)\to G^N(\T_2)/K_2}(\bar{\bb X}^2)=\hat{\bb X}^2$. Finally they construct recursively in this way $\hat{\bb X}^k$ and $\bar{\bb X}^k$ for all $k\leq N$.

At this point we see the difference with our approach. We do not define $\hat{\bb X}_t^2$ nor $\hat{\bb X}^k$ but
rather we construct $\bar {\bb X}$ step by step, namely on all $G^k(\T_n)$ with $1\leq k,n\leq N$,
first by recursion on $k$ for fixed $n$ and then by recursion on $n$;
at each step we enforce the H\"older continuity on $G^k(\T_n)$ and
the compatibility with the previous levels. This is done using the Lyons-Victoir technique,
but in a very explicit and constructive way, in particular without ever using the axiom of choice,
since we have the explicit map $\exp_{k+1}\circ\log_k \colon G^k(\T_n) \to G^{k+1}(\T_n)$ which plays
the role of the injection $i_{G/K,G}\colon G/K\to G$ in \cite[Proposition 6]{Lyons2007}.

\section{An action on branched rough paths}
\label{sss:mod}

In this section we prove \Cref{thm:intro.trans}.

Given $\gamma\in\,]0,1[$, let $N=\lfloor\gamma^{-1}\rfloor$ and denote by $\C^\gamma$ the set of collections of functions $(g^\tau)_{\tau\in\T_N}$ such that $g^\tau\in C^{\gamma|\tau|}$ and $g^\tau_0=0$ for all $\tau\in\T_N$.
It is easy to see that $\C^\gamma$ is a group under pointwise addition in $t$, that is,
\[ (g+h)^\tau \coloneq g^\tau+h^\tau. \]
As a consequence of \Cref{thm:main2}, $(g,\bar X)\mapsto g\bar X$ is an action of $\C^\gamma$ on the space of anisotropic geometric rough paths.

We use the Hairer-Kelly map $\psi$ of \Cref{lmm:HK} to induce an action of $\C^\gamma$ on branched rough paths.
Given a branched rough path $X$ and $g\in\C^\gamma$ we let $gX$ be the branched rough path defined by \[ \langle gX_{st},\tau\rangle=\langle g\bar{X}_{st},\psi(\tau)\rangle, \]
where $\bar{X}$ is the anisotropic geometric rough path given by \Cref{thm:b2ag}.
As a simple consequence of \Cref{thm:main2} we obtain
\begin{prp} \label{pr:1}
Let $X\in\BR^\gamma$.
\begin{enumerate}
\item We have $g'(g X)=(g'+g)X$ for all $g,g'\in\C^\gamma$.
\item If $(g^\tau)_{\tau\in \T_N}\in\cal C^\gamma$ is such that there exists a unique $\tau\in\T_N$
with $g^\tau\not\equiv 0$, then
\[
  \langle (gX)_{st},\tau\rangle=\langle X_{st},\tau\rangle+g^\tau_t-g^\tau_s
\]
and $\langle gX,\sigma\rangle=\langle X,\sigma\rangle$ for all $\sigma\in\T$ not containing $\tau$ as a subtree.
\end{enumerate}
\end{prp}
\begin{proof}
The first claim follows from point (1) in \Cref{thm:main2}. In order to prove the second claim, let
$g=(g^\tau)_{\tau\in \T_N}\in\cal C^\gamma$ be such that there exists a unique $\tau\in\T_N$
with $g^\tau\not\equiv 0$. Then by the property of $g$ we have
\begin{align*}
  \langle gX,\tau\rangle &= \langle g\overline X, \psi(\tau)\rangle= \langle g\overline X,\tau+\psi_{|\tau|-1}(\tau)\rangle\\
  &= \langle\bar X,\tau\rangle + \delta g^\tau + \langle g\bar X,\psi_{|\tau|-1}(\tau)\rangle
\end{align*}
where $\delta g^\tau_{st}\coloneq g^\tau_t-g^\tau_s$.
By \Cref{lmm:HK} the tree $\tau$ does not appear as a factor in any of the tensor products appearing in $\psi_{|\tau|-1}(\tau)$, hence one can recursively show that $\langle g\bar X,\psi_{|\tau|-1}(\tau)\rangle=\langle\bar X,\psi_{|\tau|-1}(\tau)\rangle$ so that the above expression becomes
\begin{align*}
  \langle gX,\tau\rangle &= \langle\bar X,\tau+\psi_{|\tau|-1}(\tau)\rangle+\delta g^\tau\\
  &= \langle X,\tau\rangle + \delta g^\tau.
\end{align*}
 For the last assertion, it is enough to note that $\sigma\in\T$ contains $\tau\in\T$ if and only if $\tau$ appears in the expression for $\psi(\sigma)$; this can be expressed more precisely by saying that $\sigma\notin T(\T_N\setminus\{\tau\})$.
But if $\sigma\in T(\T_N\setminus\{\tau\})$, then $\langle g\overline X,\psi(\tau)\rangle=\langle\overline X,\psi(\tau)\rangle$.
\end{proof}
\begin{prp}\label{pr:2}
  The action of $\C^\gamma$ on branched $\gamma$-rough paths is transitive: for every pair of branched $\gamma$-rough paths $X$ and $X'$ there exists $g\in\C^\gamma$ such that $gX=X'$.
  \label{thm:trans}
\end{prp}
\begin{proof}
  We define $g\in\C^\gamma$ inductively by imposing the desired identity.
  For trees $\tau\in\T_1=\left\{ \dtR<1>,\dotsc,\dtR<d> \right\}$ we set $g^\tau_t=\langle X'_{0t},\tau\rangle-\langle X_{0t},\tau\rangle\in C^\gamma$ so that
  \begin{align*}
    \langle gX,\tau\rangle &= \langle g\bar{X},\psi(\tau)\rangle
    = \langle g\bar{X},\tau\rangle
    = \langle \bar{X},\tau\rangle+ \delta g^\tau
    = \langle X',\tau\rangle
  \end{align*}
where $\delta g^\tau_{st}\coloneq g^\tau_t-g^\tau_s$. Suppose we have already defined $g^\tau$ for all $\tau\in\T_n$ for some $n\geq 1$, satisfying the constraints in the definition of $\C^\gamma$.
  For a tree $\tau$ with $|\tau|=n+1$ we define \[ F^\tau_{st}=\langle X'_{st},\tau\rangle-\langle\bar{X}_{st},\tau\rangle-\langle g\bar{X}_{st},\psi_n(\tau)\rangle. \]
Then
  \begin{align*}
    \delta F^\tau_{sut} &= \langle X'_{su}\otimes X'_{ut},\Delta'\tau\rangle-\langle g\bar{X}_{su}\otimes g\bar{X}_{ut},\bar{\Delta}'\psi_n(\tau)\rangle\\
    &= \langle X'_{su}\otimes X'_{ut},\Delta'\tau\rangle-\langle g\bar{X}_{su}\otimes g\bar{X}_{ut},\bar{\Delta}'\psi(\tau)\rangle\\
    &= \langle X'_{su}\otimes X'_{ut},\Delta'\tau\rangle-\langle g\bar{X}_{su}\circ\psi\otimes g\bar{X}_{ut}\circ\psi,{\Delta}'\tau\rangle\\
    &= \langle X'_{su}\otimes X'_{ut},\Delta'\tau\rangle-\langle gX_{su}\otimes gX_{ut},{\Delta}'\tau\rangle
    = 0
  \end{align*}
  by the induction hypothesis. Hence there is $g^\tau\colon[0,1]\to\R$ such that $g^\tau_0=0$ and
  \begin{equation}
    g^\tau_t-g^\tau_s=\langle X'_{st},\tau\rangle-\langle\bar{X}_{st},\tau\rangle-\langle g\bar{X}_{st},\psi_n(\tau)\rangle
    \label{eq:gdef}
  \end{equation}
  whence $g\in C^{\gamma|\tau|}$; by construction
  \begin{align*}
    \langle gX,\tau\rangle &= \langle g\bar{X},\psi(\tau)\rangle
    = \langle g\bar{X},\tau\rangle+\langle g\bar{X},\psi_n(\tau)\rangle\\
    &= \langle \bar{X},\tau\rangle+\delta g^\tau+\langle g\bar{X},\psi_n(\tau)\rangle= \langle X',\tau\rangle,
  \end{align*}
where $\delta g^\tau_{st}= g^\tau_t-g^\tau_s$. This concludes the proof.
\end{proof}

\begin{prp}\label{pr:3}
The action of $\C^\gamma$ on branched $\gamma$-rough paths is free, namely if $gX=g'X$ then $g=g'$.
\end{prp}
\begin{proof}
  This follows from the fact that by \eqref{eq:gdef} the function $g^\tau$ is defined up to a constant shift. Therefore, the condition $g^\tau_0=0$ determines $g^\tau$ uniquely.
\end{proof}

Together, \Cref{pr:1}, \Cref{pr:2} and \Cref{pr:3} imply \Cref{thm:intro.trans}.

\subsection{The BCFP renormalisation}
\label{sss:BCFP}
In \cite{Bruned2017} a different kind of modification is proposed.
There, a new decoration 0 is considered so rough paths --branched and geometric-- are over paths taking values in $\R^{d+1}$.
Recall that since branched rough paths are seen as Hölder paths taking values in the character group of the Butcher-Connes-Kreimer Hopf algebra, we may think of them as an infinite forest series of the form
\begin{equation}
  X_{st} = \sum_{\tau\in\F}\langle X_{st},\tau\rangle\tau
  \label{eq:fseries}
\end{equation}
where we regard $\tau$ as a linear functional on $\HH$, such that $\langle\tau,\sigma\rangle=1$ if $\sigma=\tau$ and zero else.
The aforementioned modification procedure then acts as a translation of the series \eqref{eq:fseries}.
Specifically, for each collection $v=(v_0,\dotsc,v_d)\colon\T\to\R^{d+1}$ an operator $M_v\colon\HH^*\to\HH^*$ is defined, such that for a $\gamma$-branched rough path, $(M_vX)_{st}\coloneq M_v(X_{st})$ is a $\gamma/N$-branched rough path.

In the particular case where $v_j=0$ except for $v_0$, the action of this operator can be described in terms of an extraction/contraction map\footnote{In \cite{Bruned2017} this map is named $\delta$ but we choose to call it $\Psi$ in order to avoid confusion with the operator defined here.} $\Psi\colon\HH\to\HH\otimes\HH$.
This map acts on a tree $\tau$ by extracting subforests and placing them in the left factor; the right factor is obtained by contracting the extracted forest and decorating the resulting node with 0.
As an example, consider
\begin{align*}
  \Psi(\dtV<ijk>) &= \1\otimes\dtV<ijk>+\dtR<i>\otimes\dtV<0jk>+\dtR<j>\otimes\dtV<i0k>+\dtR<k>\otimes\dtV<ij0>+\dtI<ij>\otimes\dtI<0k>+\dtI<ik>\otimes\dtI<0j>\\
  &\quad+\dtR<i>\dtR<j>\otimes\dtV<00k>+\dtR<i>\dtR<k>\otimes\dtV<0j0>+\dtR<j>\dtR<k>\otimes\dtV<i00>+\dtI<ij>\dtR<k>\otimes\dtI<00>+\dtI<ik>\dtR<j>\otimes\dtI<00>+\dtV<ijk>\otimes\dtR<0>.
\end{align*}
Extending $v=v_0\colon\T\to\R$ to all of $\HH^*$ as an algebra morphism it is shown that
\begin{equation}
  \langle(M_vX)_{st},\tau\rangle = \langle X_{st},(v\otimes{\id})\Psi(\tau)\rangle.
  \label{eq:BCFP}
\end{equation}
Furthermore, in this case $M_vX$ is a $\gamma$-branched rough path if coefficients corresponding to trees with decoration zero are required to satisfy the stronger analytical condition
\begin{equation}
  \sup_{0\le s,t\le 1}\frac{|\langle X_{st},\tau\rangle|}{|t-s|^{(1-\gamma)|\tau|_0+\gamma|\tau|}}<\infty,
  \label{eq:BCFPbound}
\end{equation}
where $|\tau|_0$ counts the times the decoration $0$ appears in $\tau$.
Essentially, this condition imposes that the components corresponding to the zero decoration be Lipschitz on the diagonal $s=t$.

We now show how this setting can be recovered from the results of \Cref{sss:mod}.
Let $X$ be a $\gamma$-branched rough path on $\R^{d+1}$ satisfying \eqref{eq:BCFPbound}.
Since $M_vX$ is again a $\gamma$-branched rough path, by \Cref{thm:trans} there exists a collection of functions $g\in\C^\gamma$ such that $gX=M_vX$.
Moreover, this collection is the unique one satisfying
\begin{equation}
  g_t^\tau-g_s^\tau= \langle X_{st},(v\otimes{\id})\Psi(\tau)\rangle - \langle\bar X_{st},\tau\rangle - \langle g\bar{X}_{st},\psi_{|\tau|-1}(\tau)\rangle
  \label{eq:gBCFP}
\end{equation}
for all $\tau\in\T(\R^{d+1})$ where we have used \eqref{eq:BCFP} in order to express $M_vX$ in terms of $\Psi$.
Theorem 28 in \cite{Bruned2017} ensures that the first term on the right-hand side is in $C_2^{\gamma|\tau|}$ hence $g$ is actually in $C^{\gamma|\tau|}$ as required.

The approach of \cite{Bruned2017} is based on pre-Lie morphisms and crucially on a {\it cointeraction property}, which has been explored by \cite{Calaque2011}, see in particular \cite[Lemma 18]{Bruned2017}. The cointeraction property can be used for {\it time-independent} modifications, indeed note that the functional $v$ in \cite{Bruned2017} is always constant.

Let us see why this is the case. The approach of \cite{Bruned2017} is based on a cointeraction
property studied by \cite{BHZ,Calaque2011,Foissy2016} between the Butcher-Connes-Kreimer coproduct and another
\textsl{extraction-contraction} coproduct $\delta\colon\HH\to \HH\otimes \HH$. The formula is the following
\[
({\rm id}\otimes \Delta)\delta = {\mathcal M}_{1,3}(\delta\otimes\delta)\Delta.
\]
Let us consider now a character $v\in\HH^*$.
If we multiply both sides by $(v\otimes{\rm id}\otimes{\rm id})$ and set $M_v^*=(v\otimes{\rm id})\delta\colon\HH\to \HH$ as in \cite[Proposition 17]{Bruned2017}, then we obtain
\[
\Delta\,M_v^*=(M_v^*\otimes M_v^*)\Delta,
\]
namely $M_v^*$ is a coalgebra morphism on $\HH$. Then one can define a modified rough path as
$vX\coloneq M_vX=X\circ M_v^*$. The crucial Chen property is still satisfied since
\begin{align*}
  (vX)_{st} & = (v\otimes X_{st})\delta = (v\otimes X_{su}\otimes X_{ut})({\rm id}\otimes \Delta)\delta\\
  & = (v\otimes X_{su}\otimes X_{ut}){\mathcal M}_{1,3}(\delta\otimes\delta)\Delta\\
  &= ((v\otimes X_{su})\otimes(v\otimes X_{ut}))(\delta\otimes\delta)\Delta\\
  &= ((vX)_{su}\otimes (vX)_{ut})\Delta
\end{align*}
However this does not work if $v\colon[0,1]^2\to\HH^*$ is a time-dependent character. Indeed in this case
we set $(vX)_{st}\coloneq(v_{st}\otimes{\rm id})\delta$ and we obtain
\begin{align*}
  (vX)_{st} & = (v_{st}\otimes X_{st})\delta = (v_{st}\otimes X_{su}\otimes X_{ut})({\rm id}\otimes \Delta)\delta\\
  &= (v_{st}\otimes X_{su}\otimes X_{ut}){\mathcal M}_{1,3}(\delta\otimes\delta)\Delta\\
  &= ((v_{st}\otimes X_{su})\otimes(v_{st}\otimes X_{ut}))(\delta\otimes\delta)\Delta
\end{align*}
but we can not conclude that this is equal to $((vX)_{su}\otimes (vX)_{ut})\Delta$. Our construction, as explained after
formula \eqref{eq:inte}, is not purely algebraic but is based on a (non-canonical) choice of generalized Young integrals
with respect to the rough path $X$.
Moreover our transformation group, infinite-dimensional, is much larger than that finite-dimensional group studied
in \cite{Bruned2017}.

\section{Perspectives}
In this paper we have shown that the space of branched $\gamma$-rough paths is a principal
homogeneous space with respect to the linear group $\cal C^\gamma$. This is related to the analytical
properties of the operator $\delta$ defined in \eqref{eq:delta}, which is invertible under the conditions of Gubinelli's Sewing Lemma, but not in general, and in particular not in the context
of the Chen relation on trees with low degree.

It would be now interesting to see how this action can be translated on the level of controlled paths \cite{Gubinelli2003}. The space of paths controlled by a rough path $X\in\BR^\gamma$ should be interpreted as the
tangent space to $\BR^\gamma$ at $X$, and the action on rough paths should induce an action on controlled paths. In particular it should be possible to write an action on solutions to rough differential equations.

The proof of \Cref{thm:trans}, and in particular \eqref{eq:gdef}, gives a recursive way of computing the unique $g\in\C^\gamma$ translating a given branched $\gamma$-rough path into another.
An interesting feature of the BCFP scheme is that is given in terms of a coaction so explicit calculations are somewhat easier in this more restricted case as one can compute $g^\tau$ for each tree $\tau\in\T_N$ directly by extracting and contracting subforests of $\tau$ without doing any recursions (see \eqref{eq:gBCFP}.)
However, we do not have a computational rule for an important case: suppose that $\bb X$ is branched rough path lift of a stochastic process with a.s. $C^{\gamma-}$ trajectories; it would be nice to have a way of finding $g\in\C^\gamma$ such that $g\bb X$ is centered with respect to the underlying distribution of the process, provided this is possible.
Even this last problem, namely giving precise conditions under which this centering is possible is interesting in itself.
This should be related to the notion of Wick polynomials and deformations of products as considered in \cite{EFPTZ}.

More generally, in the physics literature there are various renormalisation procedures which allow to obtain convergent iterated integrals from divergent ones by subtracting suitable “counterterms”. In the context of rough paths, implementing one of the most accepted such procedures due to Bogoliubov--Parasiuk--Hepp--Zimmermman (BPHZ) has been carried out by J. Unterberger in \cite{Unterberger2010a,Unterberger2013} by means of the Fourier normal ordering algorithm and using a technique relating the trees in the Butcher--Connes--Kreimer Hopf algebra to certain Feynman diagrams.
In our context, this could provide a canonical choice for $g\in\C^\gamma$ implementing the BPHZ renormalization procedure in a way analogous to what is done in \cite{BHZ} for Regularity Structures.

\setlength{\parskip}{0ex}
\let\bibliofont\footnotesize
\bibliographystyle{arxiv}
\bibliography{BCH}

\providecommand{\href}[2]{#2}\begingroup\raggedright\begin{thebibliography}{10}

\bibitem{BA}
G.~Ben~Arous, ``Flots et s\'{e}ries de {T}aylor stochastiques,''
  \href{http://dx.doi.org/10.1007/BF00343737}{{\em Probab. Theory Related
  Fields} {\bfseries 81} no.~1, (1989) 29--77}.

\bibitem{Boedihardjo2017}
H.~Boedihardjo and I.~Chevyrev, ``An isomorphism between branched and geometric
  rough paths,'' \href{http://arxiv.org/abs/1712.01965}{{\ttfamily
  arXiv:1712.01965 [math.PR]}}.

\bibitem{BCF}
Y.~Bruned, I.~Chevyrev, and P.~K. Friz, ``Examples of renormalized {SDE}s,'' in
  {\em Stochastic partial differential equations and related fields}, vol.~229
  of {\em Springer Proc. Math. Stat.}, pp.~303--317.
\newblock Springer, Cham, 2018.

\bibitem{Bruned2017}
Y.~Bruned, I.~Chevyrev, P.~K. Friz, and R.~Prei{\ss}, ``{A Rough Path
  Perspective on Renormalization},'' {\em {\rm to appear in} Journal of
  Functional Analysis} (2019) ,
  \href{http://arxiv.org/abs/1701.01152}{{\ttfamily arXiv:1701.01152
  [math.PR]}}.

\bibitem{Bruned2018}
Y.~Bruned, C.~Curry, and K.~Ebrahimi-Fard, ``Quasi-shuffle algebras and
  renormalisation of rough differential equations,''
  \href{http://arxiv.org/abs/1801.02964}{{\ttfamily arXiv:1801.02964
  [math.PR]}}.

\bibitem{BHZ}
Y.~Bruned, M.~Hairer, and L.~Zambotti, ``Algebraic renormalisation of
  regularity structures,''
  \href{http://dx.doi.org/10.1007/s00222-018-0841-x}{{\em Invent. Math.}
  {\bfseries 215} no.~3, (2019) 1039--1156}.
  \url{https://doi.org/10.1007/s00222-018-0841-x}.

\bibitem{Calaque2011}
D.~Calaque, K.~Ebrahimi-Fard, and D.~Manchon, ``{Two interacting Hopf algebras
  of trees: A Hopf-algebraic approach to composition and substitution of
  B-series},'' \href{http://dx.doi.org/10.1016/j.aam.2009.08.003}{{\em Adv.
  Appl. Math.} {\bfseries 47} no.~2, (2011) 282--308},
  \href{http://arxiv.org/abs/0806.2238}{{\ttfamily arXiv:0806.2238 [math.CO]}}.

\bibitem{Cartier2007}
P.~Cartier, {\em A Primer of Hopf Algebras},
  \href{http://dx.doi.org/10.1007/978-3-540-30308-4_12}{pp.~537--615}.
\newblock Springer-Verlag Berlin Heidelberg, 2007.

\bibitem{Chandra2017}
A.~Chandra and M.~Hairer, ``{An analytic BPHZ theorem for regularity
  structures},'' \href{http://arxiv.org/abs/1612.08138}{{\ttfamily
  arXiv:1612.08138 [math.PR]}}.

\bibitem{Chapoton2010}
F.~Chapoton, ``Free Pre-Lie Algebras are Free as Lie Algebras,''
  \href{http://dx.doi.org/10.4153/CMB-2010-063-2}{{\em Canad. Math. Bull.}
  {\bfseries 53} no.~3, (2010) 425--437},
  \href{http://arxiv.org/abs/0704.2153}{{\ttfamily arXiv:0704.2153 [math.RA]}}.

\bibitem{Chen1977}
K.-T. Chen, ``{Iterated path integrals},''
  \href{http://dx.doi.org/10.1090/S0002-9904-1977-14320-6}{{\em Bull. Amer.
  Math. Soc.} {\bfseries 83} no.~5, (1977) 831--880}.

\bibitem{Connes1998}
A.~Connes and D.~Kreimer, ``{Hopf Algebras, Renormalization and Noncommutative
  Geometry},'' \href{http://dx.doi.org/10.1007/s002200050499}{{\em Comm. Math.
  Phys.} {\bfseries 199} no.~1, (1998) 203--242},
  \href{http://arxiv.org/abs/hep-th/9808042}{{\ttfamily arXiv:hep-th/9808042
  [hep-th]}}.

\bibitem{Coutin2002}
L.~Coutin and Z.~Qian, ``Stochastic analysis, rough path analysis and
  fractional Brownian motions,''
  \href{http://dx.doi.org/10.1007/s004400100158}{{\em Probab. Theory Related
  Fields} {\bfseries 122} no.~1, (2002) 108--140}.

\bibitem{Curry2018}
C.~Curry, K.~Ebrahimi-Fard, D.~Manchon, and H.~Z. Munthe-Kaas, ``Planarly
  branched rough paths and rough differential equations on homogeneous
  spaces,'' \href{http://arxiv.org/abs/1804.08515}{{\ttfamily arXiv:1804.08515
  [math.CA]}}.

\bibitem{Dynkin2000}
E.~B. Dynkin, {\em Calculation of the coefficients in the Campbell--Hausdorff
  formula}, pp.~31--35.
\newblock Collected Works.
\newblock American Mathematical Society and International Press, 2000.

\bibitem{EFP2}
K.~Ebrahimi-Fard, S.~J.~A. Malham, F.~Patras, and A.~Wiese, ``The exponential
  {L}ie series for continuous semimartingales,''
  \href{http://dx.doi.org/10.1098/rspa.2015.0429}{{\em Proc. A.} {\bfseries
  471} no.~2184, (2015) 20150429, 19}.

\bibitem{EFP1}
K.~Ebrahimi-Fard, S.~J.~A. Malham, F.~Patras, and A.~Wiese, ``Flows and
  stochastic {T}aylor series in {I}t\^{o} calculus,''
  \href{http://dx.doi.org/10.1088/1751-8113/48/49/495202}{{\em J. Phys. A}
  {\bfseries 48} no.~49, (2015) 495202, 17}.

\bibitem{EFPTZ}
K.~Ebrahimi-Fard, F.~Patras, N.~Tapia, and L.~Zambotti, ``Hopf-algebraic
  deformations of products and Wick polynomials,'' {\em {\rm to appear in} Int.
  Math. Res. Not.} (2018) , \href{http://arxiv.org/abs/1710.00735}{{\ttfamily
  arXiv:1710.00735 [math.PR]}}.

\bibitem{Foi}
L.~Foissy, ``Alg{\`e}bres de Hopf combinatoires.''
\newblock \url{http://loic.foissy.free.fr/pageperso/Hopf.pdf}. Master 2 lecture
  notes.

\bibitem{Foissy2002}
L.~Foissy, ``Finite dimensional comodules over the Hopf algebra of rooted
  trees,'' \href{http://dx.doi.org/10.1016/S0021-8693(02)00110-2}{{\em J.
  Algebra} {\bfseries 255} no.~1, (2002) 89--120},
  \href{http://arxiv.org/abs/math/0105210}{{\ttfamily arXiv:math/0105210
  [math.QA]}}.

\bibitem{Foissy2016}
L.~Foissy, ``{Commutative and non-commutative bialgebras of quasi-posets and
  applications to Ehrhart polynomials},''
  \href{http://arxiv.org/abs/1605.08310}{{\ttfamily arXiv:1605.08310
  [math.RA]}}.

\bibitem{Folland1982}
G.~B. Folland and E.~M. Stein, {\em Hardy Spaces on Homogeneous Groups}.
\newblock No.~28 in Princeton mathematical notes. Princeton University Press,
  1~ed., 1982.

\bibitem{Friz2010}
P.~K. Friz and N.~B. Victoir,
  \href{http://dx.doi.org/10.1017/CBO9780511845079}{{\em {Multidimensional
  Stochastic Processes as Rough Paths}}}.
\newblock Cambridge University Press, Cambridge, 2010.

\bibitem{Gubinelli2003}
M.~Gubinelli, ``{Controlling Rough Paths},''
  \href{http://dx.doi.org/10.1016/j.jfa.2004.01.002}{{\em J. Funct. Anal.}
  {\bfseries 216} no.~1, (2004) 86--140},
  \href{http://arxiv.org/abs/math/0306433}{{\ttfamily arXiv:math/0306433
  [math.PR]}}.

\bibitem{Gubinelli2010}
M.~Gubinelli, ``{Ramification of rough paths},''
  \href{http://dx.doi.org/10.1016/j.jde.2009.11.015}{{\em J. Differential
  Equations} {\bfseries 248} no.~4, (2010) 693--721},
  \href{http://arxiv.org/abs/math/0610300}{{\ttfamily arXiv:math/0610300
  [math.CA]}}.

\bibitem{Gyurko2016}
L.~G. Gyurk{\'o}, ``Differential Equations Driven by $\Pi$-Rough Paths,''
  \href{http://dx.doi.org/10.1017/S0013091515000474}{{\em Proc. Edinb. Math.
  Soc.} {\bfseries 59} no.~3, (2016) 741--758},
  \href{http://arxiv.org/abs/1205.1832}{{\ttfamily arXiv:1205.1832 [math.CA]}}.

\bibitem{Hairer2014d}
M.~Hairer, ``{A theory of regularity structures},''
  \href{http://dx.doi.org/10.1007/s00222-014-0505-4}{{\em Invent. Math.}
  {\bfseries 198} no.~2, (2014) 269--504},
  \href{http://arxiv.org/abs/1303.5113}{{\ttfamily arXiv:1303.5113 [math.AP]}}.

\bibitem{Hairer2014}
M.~Hairer and D.~Kelly, ``{Geometric versus non-geometric rough paths},''
  \href{http://dx.doi.org/10.1214/13-AIHP564}{{\em Annales de l'Institut Henri
  Poincar\'e Probabilit\'es et Statistiques} {\bfseries 198} no.~2, (2014)
  269--504}, \href{http://arxiv.org/abs/1210.6294}{{\ttfamily arXiv:1210.6294
  [math.PR]}}.

\bibitem{Lyons1998}
T.~Lyons, ``{Differential equations driven by rough signals},''
  \href{http://dx.doi.org/10.4171/RMI/240}{{\em Rev. Mat. Iberoam.} {\bfseries
  14} (1998) 215--310}.

\bibitem{Lyons2007}
T.~Lyons and N.~B. Victoir, ``An extension theorem to rough paths,''
  \href{http://dx.doi.org/10.1016/j.anihpc.2006.07.004}{{\em Ann. Inst. H.
  Poincar{\'e} Anal. Non Lin{\'e}aire} {\bfseries 24} no.~5, (2007) 835--847}.

\bibitem{Manchon2008}
D.~Manchon,
  \href{http://dx.doi.org/https://doi.org/10.1016/S1570-7954(07)05007-3}{``Hopf
  Algebras in Renormalisation,''} in {\em Handbook of Algebra}, M.~Hazewinkel,
  ed., vol.~5, pp.~365--427.
\newblock North-Holland, 2008.

\bibitem{Munthe-Kaas2008}
H.~Munthe-Kaas and W.~Wright, ``On the Hopf Algebraic Structure of Lie Group
  Integrators,'' \href{http://dx.doi.org/10.1007/s10208-006-0222-5}{{\em
  Foundations of Computational Mathematics} {\bfseries 8} no.~2, (2008)
  227--257}.

\bibitem{Nualart2011}
D.~Nualart and S.~Tindel, ``A construction of the rough path above fractional
  Brownian motion using Volterra's representation,''
  \href{http://dx.doi.org/10.1214/10-AOP578}{{\em Ann. Probab.} {\bfseries 39}
  no.~3, (2011) 1061--1096}.

\bibitem{Reutenauer1986}
C.~Reutenauer, \href{http://dx.doi.org/10.1007/BFb0072520}{``Theorem of
  Poincare-Birkhoff-Witt, logarithm and symmetric group representations of
  degrees equal to stirling numbers,''} in {\em Combinatoire
  {\'e}num{\'e}rative}, G.~Labelle and P.~Leroux, eds., vol.~1234 of {\em
  Lecture Notes in Mathematics}.
\newblock Springer, Berlin, Heidelberg, 1986.

\bibitem{Unterberger2010a}
J.~Unterberger, ``H\"{o}lder-continuous rough paths by {F}ourier normal
  ordering,'' \href{http://dx.doi.org/10.1007/s00220-010-1064-1}{{\em Comm.
  Math. Phys.} {\bfseries 298} no.~1, (2010) 1--36}.

\bibitem{Unterberger2013}
J.~Unterberger, ``A renormalized rough path over fractional {B}rownian
  motion,'' \href{http://dx.doi.org/10.1007/s00220-013-1707-0}{{\em Comm. Math.
  Phys.} {\bfseries 320} no.~3, (2013) 603--636}.

\end{thebibliography}\endgroup
\vskip 4ex minus 3ex
\end{document}